\documentclass[12pt]{article}

\usepackage{amsmath}
\usepackage{amssymb}
\usepackage{amsthm}
\usepackage{enumitem}
\usepackage{graphicx}
\usepackage{float}
\usepackage{relsize}
\usepackage{dsfont}
\usepackage[ruled,vlined]{algorithm2e}
\usepackage[margin=1.0 in]{geometry}
\usepackage{color,soul,setspace}
\usepackage[skins]{tcolorbox}
\usepackage{hyperref}

\usepackage[sort&compress,numbers]{natbib}

\usepackage[charter,uppercase=italic,lowercase=italic,greekuppercase=italic,greeklowercase=italic]{mathdesign}
\usepackage{inconsolata}
\usepackage[T1]{fontenc}
\usepackage{microtype}
\renewcommand{\leq}{\leqslant}

\renewcommand{\tilde}[1]{\widetilde{#1}}

\renewenvironment{itemize}{
  \begin{list}{-}
    {\setlength{\parsep}{3pt}
      \setlength{\labelwidth}{24pt}
      \setlength{\itemsep}{1pt}
      \setlength{\topsep}{3pt}}}{\end{list}}

\newcommand{\bigdiamond}{\mbox{\Large$\diamond$}}
\renewcommand{\bigcirc}{\mbox{\Large$\circ$}}

\title{Symmetrized importance samplers for stochastic differential equations}

\author{Andrew Leach$^*$\thanks{$^*$Program in Applied Mathematics, University of Arizona, 617 N.~Santa Rita Ave., Tucson, AZ 85721, USA.}, Kevin K. Lin$^{*,}$\thanks{Department of Mathematics, University of Arizona,  617 N.~Santa Rita Ave., Tucson, AZ 85721, USA.}, and Matthias Morzfeld$^{*,\dagger,}$\thanks{Corresponding author.  E-mail: mmo@math.arizona.edu}}

\date{\today}

\newtheorem*{prop*}{Proposition}

\newtheorem{lemma}{Lemma}

\newcommand\ddfrac[2]{\frac{\displaystyle #1}{\displaystyle #2}}

\def\Var{\textstyle \mathop{\rm Var}}
\def \E{\textstyle \mathop{\rm E}}

\def\dt{\mbox{$\Delta t$}}
\newcommand{\tdt}{\mbox{\tiny$\Delta t$}}

\def \ep{\varepsilon}  

\def \ep{\varepsilon}
\def \O{O}  
\def \R{\mathbb{R}}
\def \dsum{\sum \limits}

\renewcommand{\phi}{\varphi} 

\DeclareMathOperator*{\argmin}{\arg\!\min}
\numberwithin{equation}{section}

\begin{document}
\maketitle

\begin{abstract}
  We study a class of importance sampling methods for stochastic differential equations (SDEs). A small-noise analysis is performed, and the results suggest that a simple symmetrization procedure can significantly improve the performance of our importance sampling schemes when the noise is not too large.  We demonstrate that this is indeed the case for a number of linear and nonlinear examples.  Potential applications, e.g., data assimilation, are discussed.
\end{abstract}



\section{Introduction}

Consider a stochastic differential equation (SDE)
\begin{equation}
	\label{eq1}
	dX_t = f(X_t)~dt + \sigma ~ dB_t~,~~X_t\in\R^D,
\end{equation}
where $f:\R^D\to\R^D$ and $B_t$ is $D$-dimensional Brownian motion.  Suppose we make noisy observations of the system at times $t=T, 2T, 3T, \cdots, JT$ 
($T>0$, fixed), obtaining a sequence of measurements $Y_j=m(X_{jT})+\eta_j,$ where $m:\R^D\to\R^d$ ($d\leq D$) is the quantity being measured (the ``observable''), $\eta_j$ are independent identically-distributed (IID) random variables modeling measurement errors and $j=1\dots J$.  What is the conditional distribution of $X_t$ for $t\in[0,JT]$ given $Y_1, Y_2, \cdots, Y_J$?  This problem of ``nonlinear filtering'' or ``data assimilation'' arises in many applications; see, e.g., \cite{chorin2013stochastic,doucet2001sequential,Bocquet2010,vanLeeuwen2009}.  
A variety of algorithms have been developed to address it, but efficient data assimilation, 
especially in high-dimensional non-gaussian problems, remains a challenge~\cite{SBM15}.

This paper concerns an approach to data assimilation known as ``particle filtering'' (see, e.g.,~\cite{doucet2001sequential} for more details) based on \emph{sampling} the conditional distributions.  We present an asymptotic analysis of certain sampling algorithms designed to improve the efficiency of particle filtering, and, based on this analysis, we propose a general way to improve their performance.  The analysis relies on taking a small-noise limit, but the algorithms do not require a small noise to operate (but may not be as efficient when the noise is not small).  We focus on {\em one} step of the filtering problem, i.e., we set $J=1$ in the above, as this is sufficient to capture the computational difficulty we wish to address.  For simplicity, we assume $\eta\sim \mathcal{N}(0,rI)$, where $r>0$ is a scalar and $I$ is the $d\times d$ identity matrix; we also assume $\sigma>0$ is a scalar.  These assumptions can be relaxed if needed.

To take one step of particle filtering, one begins by discretizing Eq.~(\ref{eq1}) using, e.g., the Euler scheme, to obtain
\begin{equation}
	\label{SDE1}
	X_{n+1} = X_{n} + \dt ~ f(X_{n}) + \sqrt{\dt} ~ \sigma \cdot \xi_{n}~,
	~~~ X_0 = x_0 \in \mathbb{R}^D,~n=0,\cdots,N-1,
\end{equation}
where $N\dt=T$, the $\xi_n$ are IID standard normal random variables.  A straightforward application of Bayes's Theorem tells us that the conditional distribution of interest satisfies
\begin{equation}
	\label{eq3}
	p(x_1,\cdots,x_N|y)\propto\exp\left(\frac1{2\sigma^2\dt}\sum_{n=0}^{N-1}\big\|x_{n+1}-x_n -
	f(x_n)\dt\big\|^2 + \frac{\|m(x_N)-y\|^2}{2r}\right)~.
\end{equation}
One then tries to design a Monte Carlo algorithm to generate discrete-time sample paths $(X_1,\cdots,X_N)$ from Eq.~(\ref{eq3}), conditioned on the observation $y$.  We refer to the distribution in Eq.~(\ref{eq3}) as the {\em target distribution}.  They are the discrete-time analogs of the conditional distributions introduced above, with $J=1$ observation.
%

Without the last term in the exponent in Eq.~(\ref{eq3}), the target distribution is just the distribution of the discretized SDE, and one can generate sample paths by carrying out the recursion in Eq.~(\ref{SDE1}).  When the last term is included, however, it is generally not feasible to sample directly from the target distribution.  A solution to this problem is {\em importance sampling:} instead of drawing samples from the target distribution, we draw sample paths $(Z_1,\cdots,Z_N)$ from an approximation $q$, usually called the ``proposal distribution''.  Any statistics we compute based on sample paths from $q$ will be biased.  We compensate for this bias by associating a weight $W^{(k)}>0$ to the $k$th sample path $(Z^{(k)}_1,\cdots,Z^{(k)}_N)$, with $\sum_kW^{(k)}=1$, so that the {\em weighted} sample paths $(Z^{(k)},W^{(k)})$ again have the correct statistics (in a sense we make precise later).


Weare and Vanden-Eijnden \cite{eijnden2012rare,eijnden2013data} proposed an algorithm for sampling distributions like Eq.~(\ref{eq3}).
They showed that their algorithm is efficient in the sense that in the limit of small dynamical and observation noise, the relative variance of the weights vanishes (see \cite{eijnden2013data} for precise definitions and statements).
The basic idea of the sampler is to look for the most likely sample path of the target distribution~(\ref{eq3}) and use this information to modify the dynamics so that samples from the proposal remain close to the target distribution. 
In this paper, by a combination of formal asymptotic analysis and numerical examples, we show that a symmetrization procedure proposed in \cite{goodman2015small} can be applied to SDEs to improve the efficiency of importance samplers.
The symmetrization and ``small noise analysis'' has also been discussed in the context of implicit sampling \cite{chorin2010implicit,morzfeld2012random}, see \cite{goodman2015small}.


While our primary motivation here is data assimilation for SDEs, our symmetrization procedure may be effective for sequential Monte Carlo sampling of more general types of systems.  As well, the class of importance sampling algorithms studied here are closely related to algorithms proposed in \cite{dupuis2004importance,dupuis2007subsolutions,dupuis2012importance,dupuis2011rare,dupuis2015escaping} and in \cite{eijnden2012rare} for sampling ``rare events'' in SDEs, though there are some significant differences between the two applications.  We plan to explore some of these connections in future work.

\paragraph{Paper organization.}  The remainder of this paper is organized as follows.
We state our main results in Section~\ref{sec:ProbStatement}.  Section~\ref{section:background} briefly reviews the linear map method and its symmetrization, as well as the small noise theory (see \cite{goodman2015small}).  We explain a new sampling method, the dynamic linear map, in Section~\ref{section:dynamic}.  We study its efficiency in the small noise regime and show how to use symmetrization to improve its efficiency in small noise problems.  Several numerical examples are provided in Section~\ref{section:numerical_examples} that illustrate our asymptotic results as well as the efficiency of our dynamic approach in multimodal problems.  The continuous time limit of the dynamic linear map is discussed in Section~\ref{section:continuous_time} and we present conclusions in Section~\ref{sec:Conclusions}.

\section{Problem statement and summary of results}
\label{sec:ProbStatement}
We now formulate the problem more precisely and summarize our key findings.  We consider a discretized SDE in the small noise regime
\begin{equation}
	\label{SDE}
	X_{n+1} = X_{n} + \dt ~ \tilde{f}(X_{n},\dt) + \sqrt{\dt} ~ \sqrt{\ep}~ \sigma
	\cdot \xi_{n}~, ~~~ X_0 = x_0 \in \mathbb{R}^D,
\end{equation}
where $\tilde{f}(x,\dt)=f(x) + \O(\dt)$ corresponds to a numerical discretization of $\dot{x}=f(x)$ (for most of this paper, we assume the Euler discretization $\tilde{f}(x,\dt)=f(x)$), and $\ep\ll 1$ is the ``small noise parameter''.
Throughout this paper we assume that the $D$-dimensional vector field $\tilde{f}$ is smooth, and that the process starts at a given initial position $x_0$ and proceeds for $N$ time steps of size $\dt$ each.  The transitions are made with independent gaussian samples $\xi_n\sim{\mathcal N}(0,I)$.  We denote the path as $x_{1:N}$, a sequence of positions $x_1,...x_N$, and its likelihood in the process with the path distribution $\rho (x_{1:N}|x_0)$.

The observation of the state at time $N\dt$ gives rise to the \emph{likelihood}
\begin{equation}
	\label{observable}
	\theta(x_N) := \exp \left(-\tfrac{1}{\ep} g(x_N) \right),
\end{equation}
where $g$ is assumed to be a smooth, nonnegative function.  
For example, for observations $y = m(x_N)+\eta$, $\eta\sim\mathcal{N}(0,\varepsilon r I)$,
we have $g(x_N) = (2r)^{-1}\,\vert\vert m(x_N)-y\vert\vert^2$.
Hereafter we will sometimes refer to $g$ as the ``log-likelihood,'' in a slight abuse of standard terminology.
%
By Bayes's Theorem, the target distribution then has the form
\begin{equation}
	\label{target}
	p(x_{1:N}|x_{0}) \propto \rho(x_{1:N}|x_{0}) \cdot
	\theta(x_N)~.
\end{equation}
Importance sampling methods generate samples using a proposal distribution $q$, and attach weights
\begin{equation}
	W^{(k)} = w(X^{(k)}_{1:N}|x_0) = p(X^{(k)}_{1:N}|x_{0}) / q(X^{(k)}_{1:N}|x_{0})
\end{equation}
to each sample, so that the weighted samples can be used to compute unbiased statistical estimates with respect to the target distribution.  To measure the efficiency of the sampling methods, we evaluate the relative variance of the weights
\begin{equation}
\label{eq:Q}
	Q := \frac{\Var\left[ W \right]}{\E \left[ W \right]^2}.
\end{equation}
Here the expected values are computed with respect to the proposal distribution $q$.
This relative variance $Q$ is connected to a standard heuristic called the ``effective sample size,'' defined by
\begin{equation}
	N_\text{eff} :=\frac{N_e}{1+Q},
\end{equation}
where $N_e$ is the number of weighted samples (see, e.g., \cite{Bergmman99,LiuChen98,doucet2001sequential}).  The effective sample size is meant to measure the size of an unweighted ensemble that is equivalent to the weighted ensemble of size $N_e$.  All else being equal, the smaller the $Q$, the more efficient the importance sampling algorithm, and if all the samples were independent, we would have $Q=0$ and $N_{\text{eff}}=N_e$.  The quantity $Q$ is convenient because it is not tied to any specific observable; recent work (see \cite{agapiou2015importance}) has also given it a more precise meaning. Other quantities that can assess effective sample sizes are discussed in \cite{MEL17}.
We note that in practice, $p$ and $q$ are only known up to a constant.  The algorithms we describe do not require knowing the normalization constants.  Likewise, $Q$ is invariant under rescaling of $p$ or $q$ by a constant.

We study two types of importance sampling methods in this paper.  The first method, called the ``linear map'' (LM), uses a gaussian proposal distribution centered at the most likely path.  The second method, called dynamic linear map (DLM), re-applies the linear map after each time step between $t=0$ and $t=N\Delta t$ given the previous moves.  Note that the linear map can be viewed as a version of implicit sampling \cite{chorin2010implicit,morzfeld2012random} applied to the path distribution of an SDE.  The dynamic linear map applies this implicit sampling step repeatedly to transition densities and is also closely linked to the continuous time control method of Weare and Vanden-Eijnden \cite{eijnden2012rare,eijnden2013data} (see also Section~\ref{section:continuous_time}).  For each method, we perform a symmetrization and exploit symmetries of the proposal distributions to increase sampling efficiency.  Symmetrization was previously studied for the LM in a more general context in \cite{goodman2015small}.  Here we adapt this procedure to problems involving SDE and to the dynamic linear map.  Following the approach taken in \cite{goodman2015small}, we show that under suitable assumptions (see Section~\ref{section:dynamic}), the relative variances of the various methods are as follows:
\begin{center}
	\begin{tabular}{c|c l}
		Method & $Q(\ep)$ scaling & \\\hline
   		\textbf{Linear Map (LM)} & $ \O(\ep)$ &\\
    	\textbf{Symmetrized LM} & $ \O(\ep^2)$ &\\
    	\textbf{Dynamic LM (DLM)} & $ \O(\ep)$ & \\
    	\textbf{Symmetrized DLM} & $ \O(\ep^2)$ & \\
	\end{tabular}
\end{center}
We also present examples showing that the leading coefficient of the DLM can be smaller than that of LM, suggesting that DLM may be more effective in some situations (see Section~\ref{section:numerical_examples}).  We discuss the continuous time limit of LM and DLM for scalar SDE, and calculate the leading coefficient of $Q(\ep)$ in an asymptotic expansion in $\ep$.  In doing so, we show that, under additional assumptions, the sampling method discussed in \cite{eijnden2012rare} is recovered in the $\dt\to 0$ limit of the DLM (see Section~\ref{section:continuous_time}).

\paragraph{Notes.}
\begin{enumerate}

\item The $\ep$-expansions we will consider are formally justified as the relevant quantities, e.g., relative weight variance, are gaussian integrals.

\item 
The insertion of the small noise parameter $\ep$ into the problem is mainly to enable asymptotic analysis.  In specific problems, there is not always an identifiable small parameter, and in any case our methods do not require a small parameter to operate.

\end{enumerate}

\section{Background}\label{section:background}
We simplify notation and write $x:=x_{1:N}$, and $F(x):=F(x_{1:N}|x_0)$,
and consider the small noise target distribution defined in \eqref{target} which can be written as $p(x)\propto \exp(-F(x)/\ep)$, where
\begin{equation}
\label{eq:F}
	F(x) =
	\frac{\dt}{2\sigma^2}\sum_{n=0}^{N-1}\Big\|\frac{x_{n+1}-x_n}{\dt} -
	\tilde{f}(x_n,\dt)\Big\|^2 + g(x_N)~,
\end{equation}
for $g$, a scalar function as in \eqref{observable}.
If we assume that $F$ has a unique, nondegenerate minimum, and let
\begin{equation}
  \label{eq:optimal-path}
  \phi = \argmin_{x \in \R^{D \cdot N}}  F(x),
\end{equation} 
i.e., $\phi$ is the optimal path with prescribed initial condition $x_0$, we can employ Laplace asymptotics to expand the target distribution around $\phi$.  (See, e.g., \cite{ostrovskii_exact_2003} for a general formulation of Laplace asymptotics.)  After a change of variables
\begin{equation}
\label{eq:ChangeOfVars}
	z = \ep^{-1/2} \cdot (x-\phi)
\end{equation}
the expansion is
\begin{equation}
	F(z) = F(\phi) + z^T H z/2+ \ep^{1/2} C_{3}(z) + \ep C_{4}(z) + \O(\ep^{3/2}),
\end{equation}
where $H$ is the Hessian evaluated at $\phi$, $C_k$ are the higher order terms in the Taylor series.  Here and below, we use the shorthand $F(z):= F(\phi + \ep^{1/2} z)$, and similarly write $w(z)$ for $w(\phi + \ep^{1/2} z)$ etc.  Note that while we will continue to refer to $z:=\{z_{1},\dots,z_n\}$ as a ``path'' after the change of coordinates, $x=\phi+\sqrt{\ep} z$ is the actual solution of Eq.~(\ref{SDE}).

The small noise analysis of LM, and other methods to follow will make frequent use of this expansion, as well as the ``variance lemma'' (see \cite{goodman2015small}).
\begin{lemma}
	{\bf{(Variance Lemma)}} For a function $u(z,\ep)$ that can be expanded in $\ep$ at least to the terms
	\begin{equation}
		u(z) = 1 + \ep^{r} u_1(z) + \ep^{2r} u_2(z) + \O(\ep^{3r})
	\end{equation}
	the relative variance of $u$ with respect to a probability density $q$ is
	\begin{equation}
		Q =  \ep^{2r}\Var_{q} \left[ u_1(z) \right] + \O(\ep^{3r})
	\end{equation}
\end{lemma}	

\subsection{Linear map}

\begin{algorithm}[tb]
  		\nl Calculate $\phi$ and $H$ starting from $x_{0}$\;
	\For{$m = 1$ to $M$}{ 
  		\nl Sample $X \sim \mathcal{N}(\phi, \ep H^{-1})$\;
		\nl Calculate $W = p(X)/q(X)$\;
	}
	\nl Return $M$ weighted samples $X,W$\;
	\caption{Linear Map}\label{figure:LM_alg}
\end{algorithm}

The proposal distribution of the linear map (LM) sampling method, summarized in Algorithm~\ref{figure:LM_alg}, is gaussian and proportional to
\begin{equation}
	\label{qlmz}
	q(z) \propto \exp \left( - z^T H z/2 \right).
\end{equation}
The weights are the ratio of target and proposal distribution,
and can be expanded as
\begin{equation}
	w(z) = 1 - \ep^{1/2} C_{3}(z) + \O(\ep).
\end{equation}
Using the variance lemma we thus find that
\begin{equation}
		Q = \ep\Var_{q} \left[ C_{3}(z) \right] + \O(\ep^{3/2}),
\end{equation}
i.e., the relative variance of the weights is linear in $\ep$
(see \cite{goodman2015small} for more details).

\subsection{Symmetrized linear map}

It is shown in \cite{goodman2015small} that the linear map can be ``symmetrized" to improve the 
scaling of $Q$ from linear to quadratic in $\ep$.
This stems from the observation that the leading order term in the weight is an odd function with respect to the random variable $z$, whose probability distribution function is even.
The symmetrized sampler uses a proposal distribution which reweights equally likely samples from the gaussian distribution of the linear map such that the resulting weights have even symmetry.  
The odd leading order terms in the weight expansions then cancel, which results in a quadratic scaling of $Q$ in $\ep$.

Specifically, the symmetrized linear map 
draws a sample $z$ from the proposal distribution $q$.
It returns $z$ with probability $w^+/(w^- + w^+)$,
and $-z$  with probability $w^-/(w^- + w^+)$,
where
\begin{equation}
	w^+ = \frac{p(-z)}{q(z)} ~~\mbox{and}~~ w^- = \frac{p(z)}{q(z)}.
\end{equation}
Samples generated in this way have a non-symmetric distribution,
but even weights:
\begin{equation}
\label{eq:symmProp}
	q_{s}(z) = q(z) \frac{2 w^+}{w^- + w^+}, \quad w_{s}(z) = \frac{w^- + w^+}{2}.
\end{equation}
The Taylor expansion of the symmetrized weight is
\begin{equation}
	w_{s}(z) = 1 + \ep \left( \tfrac{1}{2}C_{3}(z)^{2} - C_{4}(z) \right) + \O(\ep^{2}),
\end{equation}
which, together with the variance lemma shows that 
\begin{equation}
	Q_{s} = \ep^2 \Var_{q} \left[ \frac{1}{2}C_{3}(z)^{2} - C_{4}(z) \right] + \O(\ep^4).
\end{equation}
The symmetrization therefore improves the linear scaling of $Q$ in $\ep$ of LM, 
to a quadratic scaling of $Q$ for SLM
(see \cite{goodman2015small} for more details).

\section{Dynamic linear map and its symmetrization}\label{section:dynamic}


\subsection{A multimodal example}
\label{sec:A multimodal example}

The linear map can be efficient when the hypotheses underlying its derivation are satisfied, i.e., when the pathspace distribution is unimodal and a gaussian approximation is appropriate.  However, when there are multiple modes, LM can become inefficient.  To see how this might happen, consider the simple random walk
\begin{equation}
  X_{n+1} = X_n + \sqrt{\dt}~\sqrt{\ep}~\xi_n
\end{equation}
i.e., $X_{n} = X_0 + \sqrt{\dt}~\sqrt{\ep}~W_n$ where $W_n$ is standard Wiener process.  Suppose we have a bimodal likelihood function $e^{-g(x)/\ep}$ whose graph is as shown in Figure~\ref{fig:cond-bm}; this type of situation can arise when multiple states can give the same measurement, so that observations may have ambiguous interpretation.  In this case, the high probability paths will be those that reach the vicinity of $x=\pm1$ at $t=1$; effectively, the high probability paths are sample paths of Brownian motion, conditioned to be near $x=\pm1$ at $t=1$.  The probability of this occurring by chance is exponentially small as $\ep\to0$, and direct sampling is unlikely to ever produce such a path.

\begin{figure}
  \begin{center}
    \resizebox{3.5in}{!}{\includegraphics{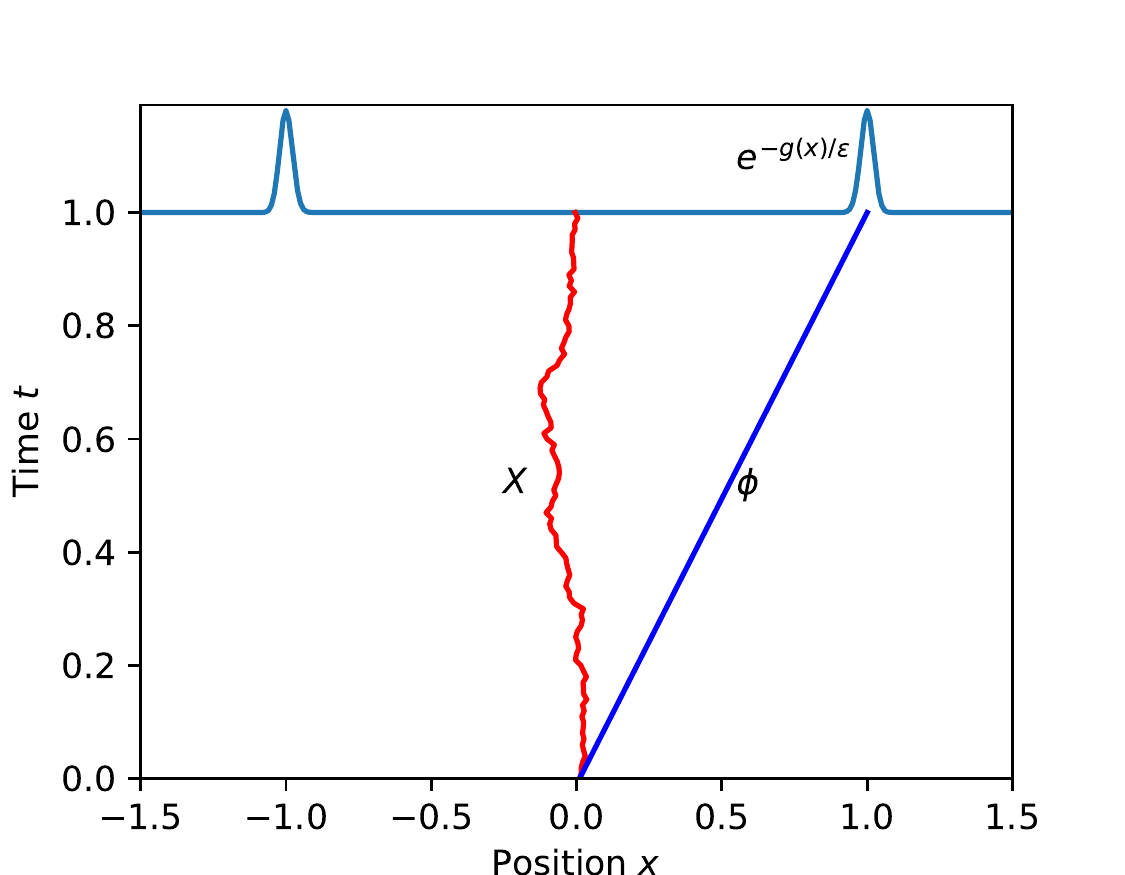}}
  \end{center}
  \caption{Brownian motion with bimodal likelihood.  Here, the initial condition is $X_0=0.01$, and we use $\ep=0.1$~.  Shown are a sample path $X$ and the optimal path $\phi$ starting from $X_0$.}
  \label{fig:cond-bm}
\end{figure}

A straightforward calculation shows that the optimal path $\phi$ approaches a straight line in the $xt$-plane as $\ep\to0$, going to the right bump if $X_0>0$, to the left if $X_0<0$ (and undefined if $X_0=0$).  With a bimodal likelihood function, the target distribution $p(x)$ is bimodal as well.  If the initial condition is sufficiently to the right of $x=0$, one of the two modes will dominate, and LM can be expected to be effective.  As $X_0$ moves closer to $x=0$, however, the other mode will begin to make a greater contribution; at $X_0=0$, the two modes carry exactly the same weight.  But LM will \emph{always} pick the mode on the right when $X_0>0$, no matter how close $X_0$ is to $x=0$.  So LM will produce essentially no sample paths going to the left, leading to a large weight variance.  See Section~\ref{section:numerical_examples} for detailed numerical results.

This is a well-known problem with importance sampling algorithms.  Similar issues arise in rare event simulation, and a standard solution is to dynamically recompute the optimal path.  See, e.g., the discussion of Siegmund's algorithm in \cite{asmussen_stochastic_2007}.  In our context, this leads to an algorithm we call the dynamic linear map, which is similar to the algorithms proposed in \cite{eijnden2012rare,dupuis2007subsolutions}.  We will also discuss symmetrization in this context.

\subsection{Dynamic linear map}
\label{section:Dynamic linear map}


Roughly speaking, the dynamic linear map (DLM) consists of computing the optimal path $\phi$ starting from the current state $X_n$, taking \emph{one} step (so that $X_{n+1}=\phi_{n+1}$), then repeating.  See Algorithm~\ref{figure:DLM_alg} for details.  The DLM thus requires redoing LM \emph{at every step}, and is therefore 
more expensive.\footnote{Suppose each cost function evaluation requires CPU time $\propto N$, the number of steps, and each optimization requires $k$ function evaluations.  Then all else being equal, LM has running time $O(kN)$ and DLM $O(kN^2)$.}  However, it can avoid some of the issues arising from multi-modal target distributions. One can see this heuristically in the above example (Section~\ref{sec:A multimodal example}): suppose we start with $X_0$ slightly to the right of $x=0$, so that the optimal path $\phi$ goes to the right bump.  After a few steps, we may end up in a state $X_n$ closer to the left bump.  At this point, the DLM would start steering the sample path towards the left bump.  Unlike LM, repeated sampling using DLM would yield sample paths that end at both the left and the right bumps (see Section~\ref{sec: Brownian motion and linear SDE}).

To make use of DLM, we need an expression for the associated weights.  This, in turn, requires an expression for the proposal distribution $q$ associated with DLM, which one can derive by first noting that in general, transition densities are marginals of the pathspace distribution:
\begin{align*}
	\rho(x_{n+1}|x_n)&= \int \rho(x_{n+1:N}|x_n)~dx_{n+2:N} .
\end{align*}
(Here we abuse notation slightly and use $p$ and $q$ to denote both pathspace distributions as well as their marginals.)  The DLM transition density arises from making a gaussian approximation of the target distribution at each step, then taking its marginal.  This leads to
\begin{align}
\label{eq:DLMProposal}
	q(x_{n+1}|x_n)&= \int q(x_{n+1:N}|x_n)~dx_{n+2:N}\\
	&\propto \exp \left (- (x-\phi)_{{\mbox{\tiny$n+1$}}}^T \Sigma_{n+1}^{-1} (x-\phi)_{{\mbox{\tiny$n+1$}}}/(2 \dt)\right).\nonumber
\end{align}
Here $\phi$ is the optimal path from $x_n$ to $x_N$
and we omit its dependence on $x_n$ for readability of the equations;
we also remind the reader that $x=x_n,\dots,x_N$ is a path.
We denote the Hessian of $F(x)$ evaluated at the optimal path $\phi$ by $H$.
We view a path from $x_n$ to $x_{n+k}$ as a point in $\R^{kD}$, arranged in $k$ blocks of $D$ entries.  Accordingly, the matrix $H$ can be viewed as an element of $\R^{(N-n)D\times(N-n)D}$ 
and can be subdivided into $(N-n)\times(N-n)$ blocks of dimension $D\times D$ each.  The matrix $\Sigma_{n+1}$ in Eq.~(\ref{eq:DLMProposal}) is $(H^{-1})_{1,1}/\dt$, the first block of the inverse of the Hessian $H$ (after rescaling).

\begin{algorithm}[tb]
	\For{$m = 1$ to $M$}{
		\For{$n = 0$ to $N-1$}{
			\nl Calculate $\phi$ and $H$ starting from $X_{n}$\;
			\nl Calculate $\Sigma_{n+1} = (H^{-1})_{\mbox{\tiny{1,1}}}/\dt$\;
	  		\nl Sample $X_{n+1} \sim \mathcal{N}(\phi_{\mbox{\tiny{n+1}}}, ~\dt ~\ep ~ \Sigma_{n+1})$\;
			\nl Calculate $W_{\mbox{\tiny{n}}} = p(X_{\mbox{\tiny{n+1}}}|X_{\mbox{\tiny{n}}})/q(X_{\mbox{\tiny{n+1}}}|X_{\mbox{\tiny{n}}})$\;
		}
	}
	\nl Calculate $W = W_{\mbox{\tiny{N-1}}} \cdot ... \cdot W_{\mbox{\tiny{0}}}$\;
	\nl Return $M$ weighted samples $X,W$\;
	\caption{Dynamic Linear Map}\label{figure:DLM_alg}
\end{algorithm}
In Algorithm~\ref{figure:DLM_alg}, going from step $n$ to $n+1$ requires optimizing over the $(N-n)D$ remaining variables in the path.  This is done independently at every step and for every sample path.  The weights for the proposal distribution of DLM can be calculated as described in Algorithm~\ref{figure:DLM_alg}, or as the product of the incremental weights
\begin{equation}
	w = \prod_{n=0}^{N-1} w_n,\quad 
	w_n\propto \frac{p(x_{n+1}\vert x_{n})}
	{q(x_{n+1}\vert x_n)}.
\end{equation}

\medskip
\noindent
\emph{Relation to Hamilton-Jacobi equation and regularity of ``value functions''.}  In the definitions above, it is assumed that $q(x_{n+1}|x_n)$ is well-defined for all $(x_n,x_{n+1})$.  This is actually not always the case.  To see this, consider again the example from Section~\ref{sec:A multimodal example}.  If $x_n=0$ at some $n$, there are two optimal paths pointing in opposite directions.  At this point, because there is not a single optimal path, $q(x_{n+1}|x_n)$ is undefined.  This behavior is actually rather common, and not at all confined to the Brownian motion example.  It is closely connected with regularity of solutions of a partial differential equation of Hamilton-Jacobi (HJ) type.  As we do not make use of the theory of HJ equations in this paper, we do not go into details here.  Instead, we provide a brief summary below, and refer interested readers to, e.g., \cite{eijnden2013data} or \cite{dupuis2004importance,dupuis2007subsolutions,dupuis2012importance,dupuis2011rare}, for more information.


In the DLM method, the optimal path minimizes a version of the function $F$ in Eq.~(\ref{eq:F}), but starting with state $x_n$ at time $n$ rather than always at time 0.  In the limit as $\dt\to0$, the \emph{value function} $u(x,t)$ achieved with initial condition $x_n=x$ at step $n\dt=t$ solves a HJ equation of the form $\partial_tu = H(x,Du)$, with Hamiltonian $H(x,p) = \frac{\sigma^2}2|p|^2 + p\cdot f(x)$; this is the Legendre transformation of the Freidlin-Wentzell Lagrangian $L(x,v) = \frac1{2\sigma^2}|v-f(x)|^2$~\cite{freidlin1984random}.  For the HJ equation to be well-posed, one prescribes the \emph{final condition} that $u(x,T)=g(x),$ where $g$ is the likelihood in Eq.~(\ref{observable}) and $T>0$.  The HJ equation is then solved backwards in time.  The time derivative $\dot\phi$ of the optimal path starting at position $x$ and time $t$ is given by the gradient of $u(x,t)$ where it is differentiable.  At locations $(x,t)$ where there are multiple optimal paths, the value function $u(x,t)$ is generally continuous but not differentiable.  At such \emph{singular points} $x$, $q(x_{n+1},x)$ has jump disconinuities (as $x$ varies) and is therefore undefined.

Though very much relevant to the efficacy of the type of methods discussed in this paper, the analysis of singularities of HJ equations can be highly nontrivial.  As our main goal is to assess whether some version of the symmetrization procedure proposed in \cite{goodman2015small} can be extended to SDEs, we have opted to focus on the simplest possible setting, leaving more general analysis to future work.  \emph{For the remainder of the paper, we make the following {\bf standing assumption:}}
\begin{displaymath}
  \mbox{$q(x_{n+1}|x_n)$ is defined everywhere, and is as smooth as needed.}
\end{displaymath}
The analytical results described below should therefore be interpreted as a \emph{best-case scenario.}  We also note that while the numerical algorithm is unlikely to produce an $x_n$ \emph{exactly} in the set of singular points in actual practice, the presence of singularities does mean that the performance of the algorithm may be worse than predicted by our analysis.  We have therefore designed our numerical examples to test the extent to which the algorithms behave as predicted even when $q(x_{n+1}|x_n)$ is not differentiable everywhere.

\subsection{Small-noise analysis}


To find the scaling of the relative variance of the weights of DLM with the small noise parameter $\ep$, we apply the same change of variables as in Eq.~\eqref{eq:ChangeOfVars} to each transition density and expand the incremental weights $w_n$ as
\begin{equation}
\label{eq:weightsDLM}
w_n ~~=~~ w(z_{n+1}|z_n) ~~=~~ 1 + \ep^{1/2} \cdot w_{1,n}(z_{n+1}|z_n) + \ep \cdot w_{2,n}(z_{n+1}|z_n) + \O(\ep^{3/2}),
\end{equation}
where
\begin{align}
	w_{1,n}(z_{n+1}|z_n) &= \ddfrac{\int C_{3}(z) \exp \left (- z^T H z/2 \right )~dz_{n+2:N}}{\int \exp \left (- z^T H z/2 \right )~dz_{n+2:N}} \\[1ex]
	w_{2,n}(z_{n+1}|z_n) &= \ddfrac{\int (C_{3}(z)^2/2 - C_{4}(z)) \exp \left (- z^T H z/2 \right )~dz_{n+2:N}}{\int \exp \left (- z^T H z/2 \right )~dz_{n+2:N}} \nonumber \\
	& ~~~- \int (C_{3}(z)^2/2 - C_{4}(z)) \exp \left (- z^T H z/2 \right )~dz_{n+1:N},
\end{align}
%
noting that Eq.~(\ref{eq:weightsDLM}) relies strongly on our standing assumption that $q(x_{n+1}|x_n)$ is differentiable.  Since the weight of a sample is the product of the incremental weights, we have
\begin{align*}
  w(z) &= 1 +\ep^{1/2} \cdot w_{1} + \ep \cdot w_{2} + \O(\ep^{3/2}),
\end{align*}
where
\begin{equation}	
  w_{1} = \dsum_{n=0}^{N-1} w_{1,n}, \quad
  w_{2} = \dsum_{n=0}^{N-1} w_{2,n} + \dsum_{n=0}^{N-1} \dsum_{m=0}^{N-1} w_{1,n} \cdot w_{1,m}.
\end{equation}
The scaling of $Q$ in $\ep$ now follows from the variance lemma:
\begin{equation}
	Q^{\ep} = \ep \cdot \Var_{q} \left[ w_1 \right] + \O(\ep^2).
\end{equation}
Thus, the relative variance of DLM scales linearly in $\ep$, the same asymptotic scaling as LM.
However, we will show in numerical examples below that  
the dynamic approach can be more effective in practice than LM, 
especially when the target distribution has multiple modes.

\subsection{Symmetrization}\label{section:symmetrization}

The leading order term in the weight for DLM has an odd symmetry, just like the LM, and a symmetrization procedure can be applied to DLM to improve the scaling of $Q$ in $\ep$.  The reason is that, at each time step, $X_{n+1}$ is generated by a composition of the previous state $X_n$ and a new gaussian sample $\xi_n$.  While this procedure leads to a proposal distribution that is not necessarily even, the paths are constructed incrementally from gaussian samples which are even.

More specifically, 
the recursive composition forms a map $h$ from the $N\cdot D$ dimensional gaussian to the path $X = h(\ep^{1/2} \xi)$,
and for every sampled path $X^{+} = h(\ep^{1/2} \xi)$,
there is a path $X^{-} = h(-\ep^{1/2}\xi)$ which is equally likely.
Following the algorithm described in Algorithm~\ref{figure:S_alg}, we sample $X^{+}$ with probability $W^+/(W^+ +W^-)$, and $X^{-}$ with probability $W^-/(W^+ +W^-)$, the resulting proposal is a ``symmetrized'' distribution with even weights (see Eq.~\eqref{eq:symmProp}).
\begin{algorithm}[tb]
	\For{$m = 1$ to $M$}{
		\nl Sample $\xi \sim \mathcal{N}(0,I)$\;
		\nl Calculate $X^{+} = h(\ep^{-1/2} \xi)$ and $X^{-} = h(-\ep^{-1/2} \xi)$\;
		\nl Calculate $W^{+} = p(X^+)/q(X^+)$ and $W^{-} = p(X^-)/q(X^-)$\;
		\nl Sample $X = X^{+}$ with prob. $\frac{W^+}{W^+ + W^-}$ and $X = X^{-}$ with prob. $\frac{W^-}{W^+ + W^-}$\;
		\nl Calculate $W = \frac{W^+ + W^-}{2}$\;
	}
	\nl Return $M$ weighted samples $X,W$\;
	\caption{Symmetrization}\label{figure:S_alg}
\end{algorithm}

The symmetrized weights can be written in terms of the map as
\begin{equation}
	w_s(h(\ep^{1/2}\xi)) = \frac{w(h(\ep^{1/2}\xi)) + w(h(-\ep^{1/2}\xi))}{2}.
\end{equation}
Recall the expansion of the weights in \eqref{eq:weightsDLM}, and note that
\begin{align*}
		z &= \ep^{-1/2}(h(\ep^{1/2}\xi)-h(0)),
\end{align*}
since the most likely path $\phi$ can be written in terms of the map
as $\phi  = h(0)$.

If $\phi$ is unique (at each time step), $h$ 
can be expanded around the most likely path as
\begin{align}
	\label{hex1}
	h(\ep^{1/2} \xi) &= \phi +  \ep^{1/2} (Dh)(0)\cdot \xi + \O(\ep),\\
	\label{hex2}
	h(-\ep^{1/2} \xi) &= \phi -  \ep^{1/2} (Dh)(0)\cdot \xi + \O(\ep).
\end{align}
We thus have that
\begin{align}
	w(h(\ep^{1/2}\xi)) &= 1 + \ep^{1/2} w_1(\ep^{1/2} (Dh)(0)\cdot \xi,\phi) + \ep w_2(\ep^{1/2} (Dh)(0)\cdot \xi,\phi) + \O(\ep^{3/2}) \\
	w(h(-\ep^{1/2}\xi)) &= 1 - \ep^{1/2} w_1(\ep^{1/2} (Dh)(0)\cdot \xi,\phi) + \ep w_2(\ep^{1/2} (Dh)(0)\cdot \xi,\phi) + \O(\ep^{3/2})
\end{align}
which results in the cancellation of the leading order term in $\ep$ of the symmetrized weight
\begin{equation}
	w_s(h(\ep^{1/2}\xi)) = 1 + \ep w_2(\ep^{1/2} (Dh)(0)\cdot \xi,\phi) + \O(\ep^{3/2})
\end{equation}
Applying the variance lemma completes the proof for the quadratic scaling of $Q_s$ in $\ep$

\begin{equation}
	Q_s = \ep^2 \cdot \Var_{q_s}[w_2] + \O(\ep^4).
\end{equation}

\section{Numerical examples}\label{section:numerical_examples}


We now examine a number of concrete examples, both to illustrate the scaling of the proposed algorithms and to test their limitations.  The source code for all examples in this section can be found at {\tt\href{https://github.com/AndrewLeach/SDE_Importance_Sampling}{https://github.com/AndrewLeach/SDE\_Importance\_Sampling}}~.

\subsection{Examples with linear SDE}
\label{sec: Brownian motion and linear SDE}

We begin with the Brownian motion example from Section~\ref{sec:A multimodal example}:
\begin{equation}
\label{eq:BM}
X_{n+1} = X_n + \sqrt{\dt}~\sqrt{\ep}~\xi_n~,
\end{equation}
with initial condition $X_0 = x_0$ and with likelihood $\theta=e^{-g(X_N)/\ep}$ for two different choices for $g$. 
We first consider the case of a unimodal target distribution
for which the assumptions made during the small noise analysis are satisfied.
We then violate the assumption of a unique optimal path
to indicate limitations of DLM and our small noise analysis.
For the examples below, the time step is $\dt = 10^{-2}$.  The observation is collected at step 
$N = 100$ (i.e., $T=1$).  Computing the optimal paths is straightforward to do analytically
and we use the analytic formulas in our implementation of the various samplers.  

\paragraph{Brownian motion with unimodal likelihood.} 
We first consider a likelihood defined by
\begin{displaymath}
g(x) = \frac{x^4}{24} + \frac{x^3}{6} + \frac{x^2}{2}.
\end{displaymath}
The likelihood is asymmetric in $x$
and leads to a non-gaussian and unimodal target distribution.  
In this example, the assumptions made in our small noise analysis are satisfied.

We apply LM, SLM, DLM, and SDLM to sample the target distribution 
over a wide range of $\ep$, and compute the relative variance $Q$ for each of these methods.  
For each $\ep$ and method (LM, SLM, DLM and SDLM), we draw $1200$ samples .
The results are shown in Figure~\ref{Q_BM_A}.
\begin{figure}[tb]
	\centering
	\includegraphics[width = .5 \textwidth]{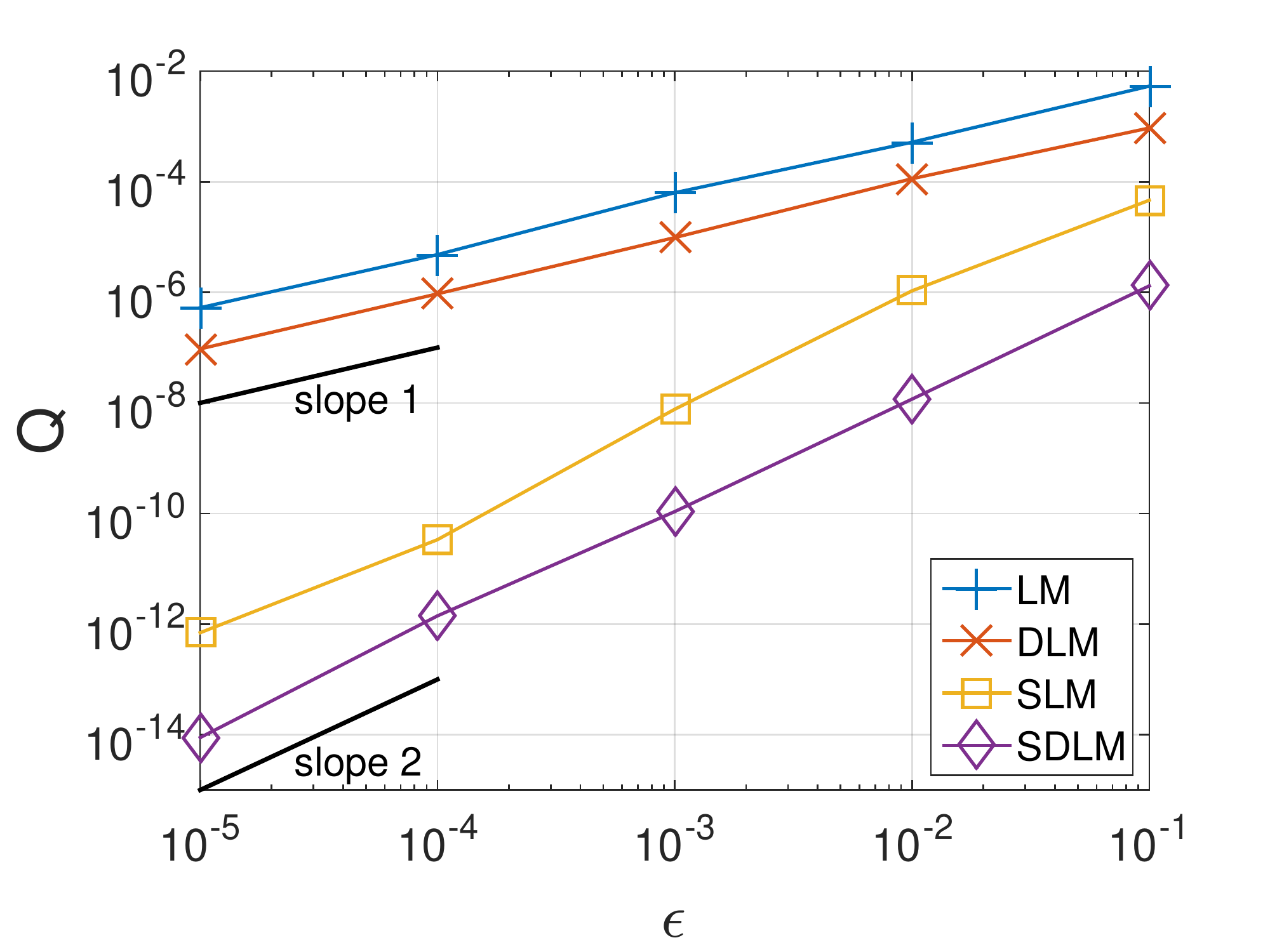}
	\caption{Brownian motion with asymmetric unimodal likelihood.  The scaling of $Q$ in $\ep$ for LM, SLM, DLM and SDLM are plotted.}
	\label{Q_BM_A}	
\end{figure}
As can be seen, the results show the predicted scalings for $Q$ for a wide range of $\ep$ for all four methods: both LM and DLM are $O(\ep)$, while SLM and SDLM are both $O(\ep^2)$.  Perhaps this is no surprise, as all assumptions that lead to the small noise theory are valid in this example.  We also see that the dynamic methods (DLM and SDLM) have smaller relative variance $Q$ at each value of $\ep$, though they also cost more per sample.


\paragraph{Brownian motion with bimodal likelihood.}
Next, we examine
\begin{displaymath}
g(x) = 100 \cdot \left(\frac{x^4}{4} - \frac{x^2}{2}\right).
\end{displaymath}
As explained in Section~\ref{sec:A multimodal example}, this leads to a bimodal target distribution.
%
%
We fix $\ep=10^{-1}$, and leave all other parameters as above.  We apply LM and DLM to compute the final-time distribution $p(X_N|X_0)$, using $1.2\times10^4$ (weighted) samples.  The results are shown in Figure~\ref{BM_WHist}, along with the target distribution $\propto e^{-(g(x) + x^2/2)/\ep}$.
\begin{figure}[tb]
  \begin{center}
    \begin{tabular}{cc}
      \resizebox{3in}{!}{\includegraphics{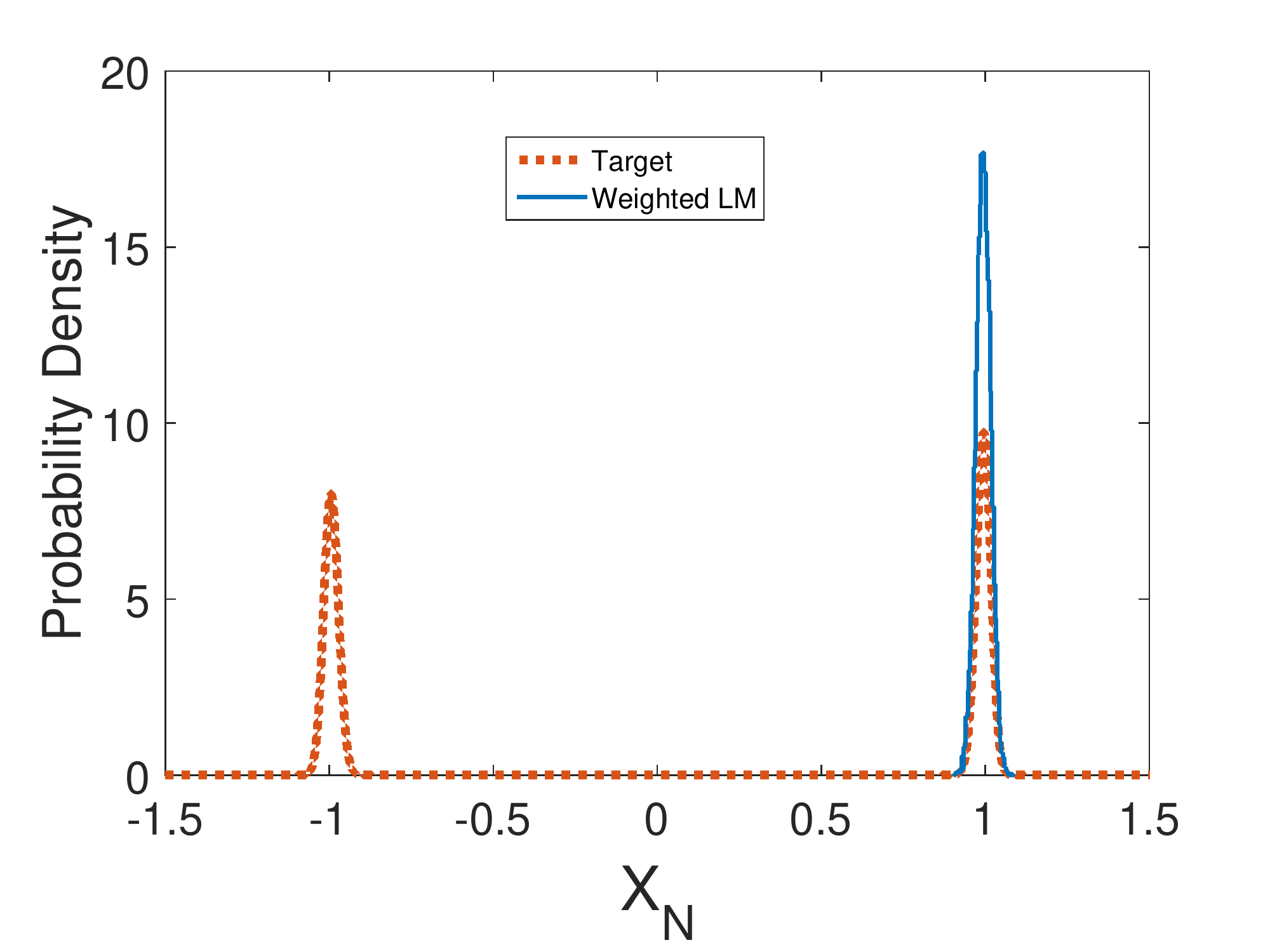}}&
      \resizebox{3in}{!}{\includegraphics{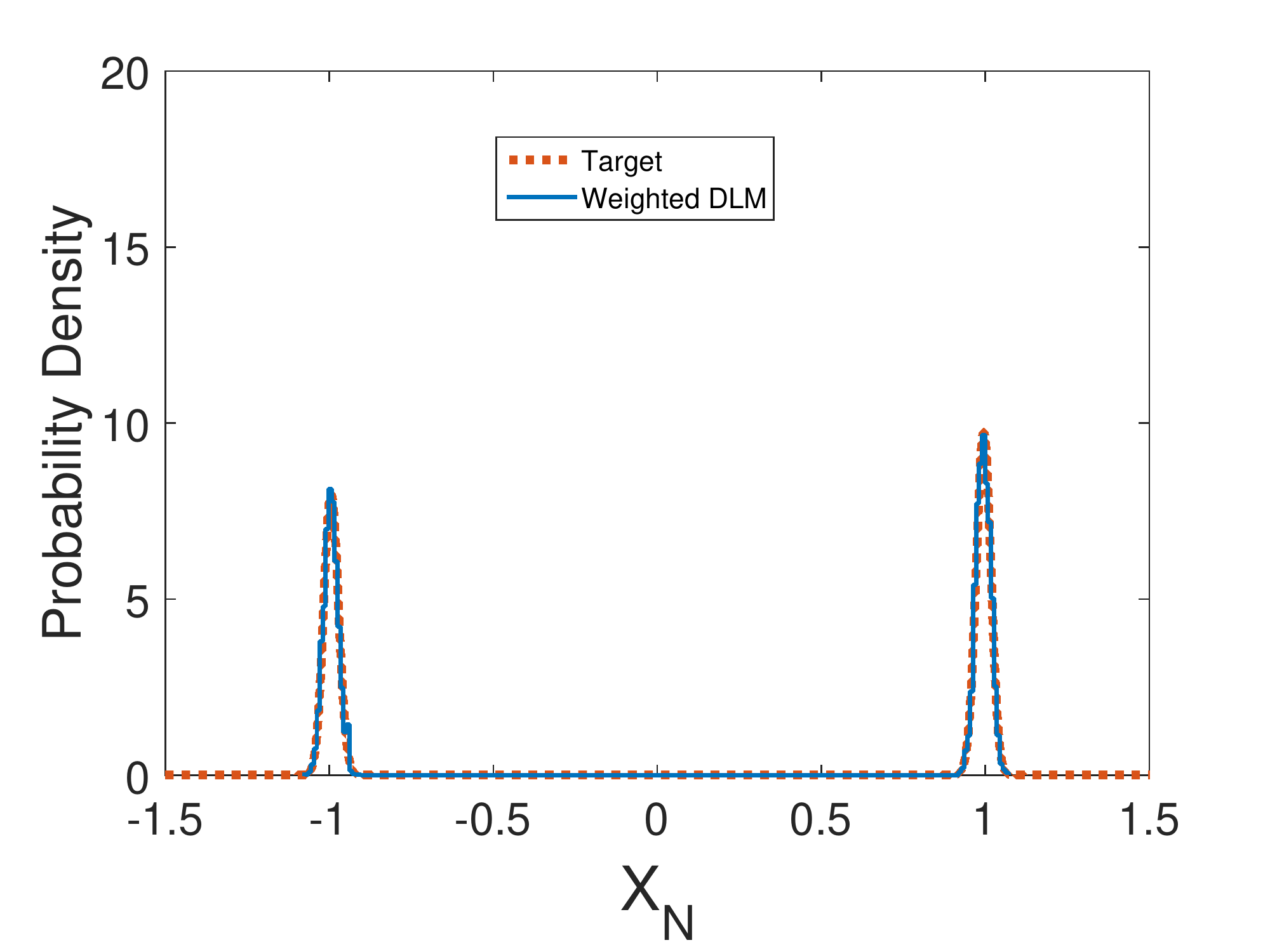}}\\
      (a) LM & (b) DLM\\
    \end{tabular}
  \end{center}

  \caption{Final-time distribution for Brownian motion with bimodal likelihood.  In (a), we plot the marginal distribution $p(x_N|x_0)$ estimated by weighted histograms of $12000$ samples generated using LM.  Also shown is the target distribution.  In (b), we plot the same information for DLM.}
  \label{BM_WHist}
\end{figure}

As expected, LM essentially ignores one of the two modes, while DLM captures both modes.
%
%
As explained before, even though both samplers should reproduce the target distribution in the large-sample-size limit, in practice LM produces almost no sample paths that go to the left bump.  In contrast, DLM readily generates sample paths ending at both bumps, leading to a more effective sampling of the target distribution.  We have experimented with increasing the sample size for LM, but even the largest sample sizes we consider did not lead to weighted samples that represent both modes.

Finally, note that empirical estimates of $Q$ are insufficient to detect this problem: even though the true value of $Q$ for LM should be quite large in this case, empirical estimates of $Q$ for LM are actually quite small because none of the sample paths go to the left bump.  Indeed, for Figure~\ref{BM_WHist}, the empirical $Q$ for LM is $\sim3\times10^{-3}$, while that of DLM is $\sim1$.  The example thus shows that for non-gaussian and possibly multimodal distributions, DLM can be more reliable despite the same scaling of $Q$.

\paragraph{Overdamped Langevin equation with bimodal likelihood.}
The scaling arguments for DLM and its symmetrized version rely on the assumption that 
the most likely path $\phi$ is unique at every time step.  We now consider an example for the DLM in which we deliberately violate this assumption. 
The model is
\begin{equation}
  \label{eq:LinearEx}
  X_{n+1} = X_n -\dt ~\alpha \cdot X_n + \sqrt{\dt}~\sqrt{\ep}~\xi_n~.
\end{equation}
This is the Euler discretization of the overdamped Langevin equation $\dot{X} = -\alpha X + \sqrt\ep~\dot{B}.$ We use the log-likelihood
\begin{displaymath}
g(x) = 10 \cdot \left(\frac{x^4}{4} - \frac{x^2}{2}\right).
\end{displaymath}
As in the previous example, the optimal path goes to the right bump when $X_0>0$ and to the left when $X_0<0$.  At $X_0=0$ there is no unique optimal path.

The linear drift makes it likely that DLM sample paths encounter the $x=0$ line
and the small noise results may not hold in this case.
To illustrate the behavior and efficiency of the methods in this situation
we perform experiments with varying values of $\ep$ and $x_0$.  Specifically, for a fixed $\ep$, we take $N=10^3$ time steps with DLM, starting from initial conditions ranging from $x_0=10^{-1}$ to $x_0=10^{-5}$.  
We compute the averaging number of $x=0$ crossings for each experiment.  Figure~\ref{Langevin_Cross} shows the results as well as the computed values of $Q$.
\begin{figure}[tb]
  \begin{center}
    \begin{tabular}{cc}
      \resizebox{3in}{!}{\includegraphics{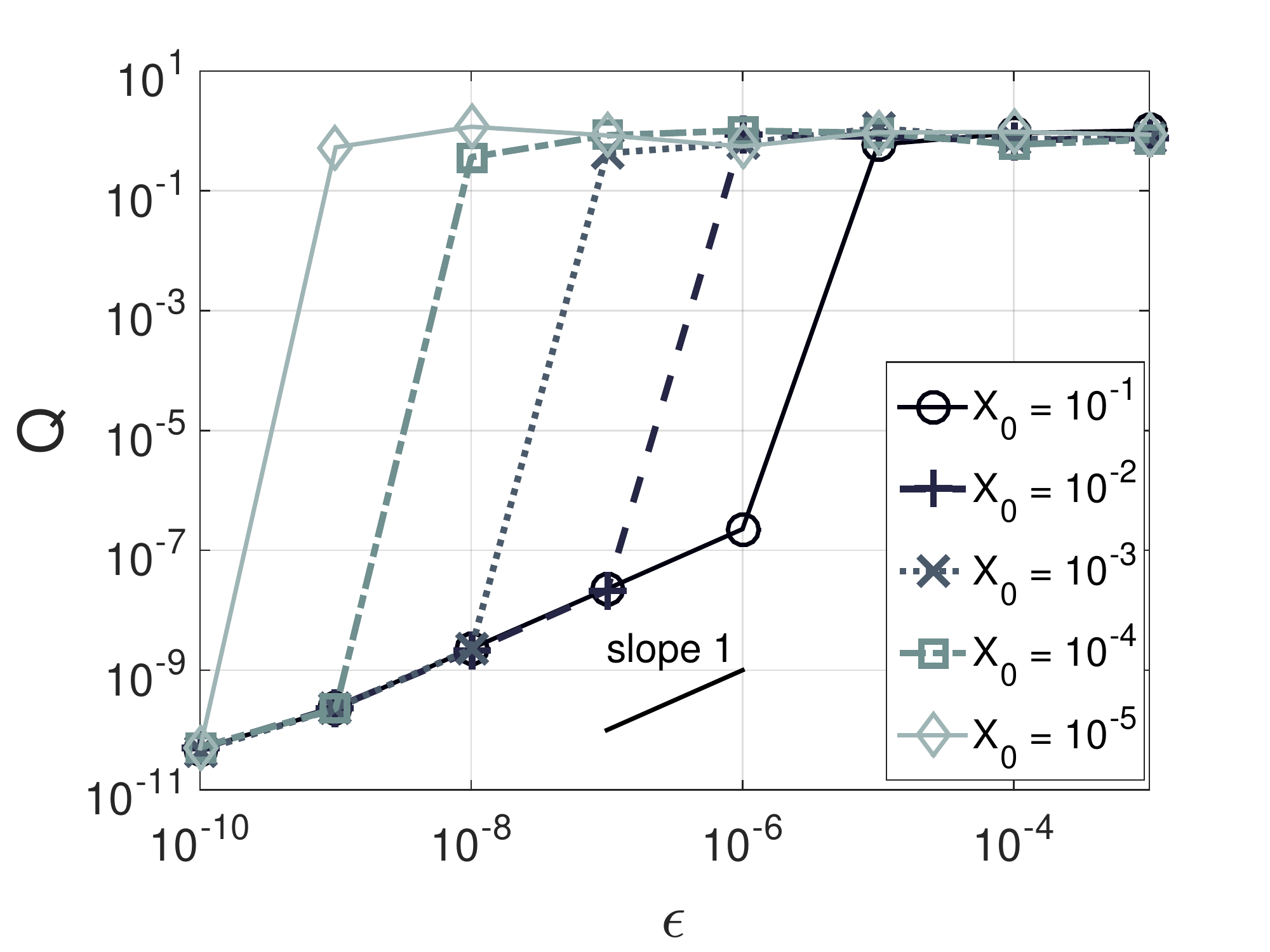}} &
      \resizebox{3in}{!}{\includegraphics{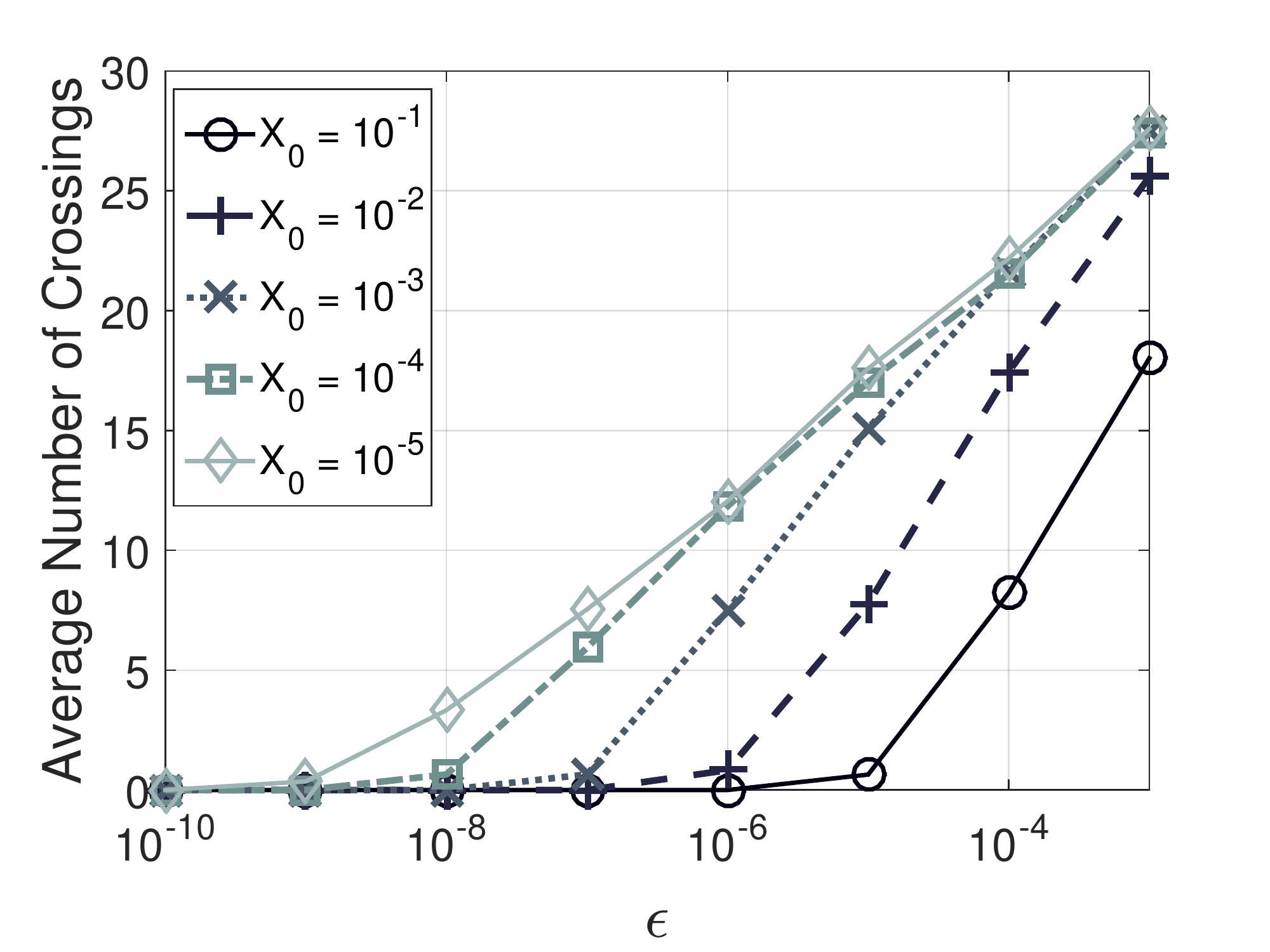}} \\
      (a) & (b)\\
    \end{tabular}
  \end{center}
  \caption{DLM applied to the overdamped Langevin equation with bimodal likelihood.  Panel (a) shows the scaling of $Q$ vs.~$\ep$ for $x_0$ approaching $x=0$.  In (b), we plot the average number of $x=0$ crossings against $\ep$.}
  \label{Langevin_Cross}
\end{figure}

As can be seen in Figure~\ref{Langevin_Cross}(a), the predicted asymptotic scaling of $Q$ only emerges for small $\ep$; the critical value of $\ep$ at which the $Q$ curve crosses over into the asymptotic regime decreases as $x_0$ approaches $0$, making crossings more likely.  Comparing Figures~\ref{Langevin_Cross}(a) and \ref{Langevin_Cross}(b), we see that the asymptotic regime corresponds to values of $\ep$ small enough that the average number of  crossings per sample is near zero.  Closer examination of the data suggests that this critical $\ep$ scales roughly linearly with distance of the initial condition $x_0$ to $x=0$.  The example thus suggests that the efficiency of DLM may suffer if one encounters non-unique optimal paths while constructing the proposal distribution $q$ sequentially, 
but the predicted $Q$ scaling again holds if $\ep$ is small enough.

Finally, we note that even in the pre-asymptotic regime, the value of $Q$ are $O(1)$, meaning the effective number of samples is $\approx N_e/2$, which is still a significant improvement over direct sampling.


\subsection{Example with a nonlinear SDE}
Our second example is a stochastic version of an idealized geomagnetic pole reversal model due to Gissinger~\cite{gissinger2012new}:
\begin{equation}
  \label{G12}
  \begin{array}{rcrrlrl}
    \dot{x}^1 &=& 0.119x^1 &-& x^2 x^3 &+& \sqrt\ep\dot{B}^1\\
    \dot{x}^2 &=& -0.1x^2 &+& x^1  x^3  &+& \sqrt\ep\dot{B}^2\\
    \dot{x}^3 &=& 0.9 - x^3 &+& x^1 x^2 &+& \sqrt\ep\dot{B}^3\\
  \end{array}~.
\end{equation}
(In this section, $x^k$ refers to the $k$th component of a vector $x$.)  The $\ep=0$ system of ordinary differential equations has 3 unstable fixed points: $(0,0,0.9)$ and $p_\pm\approx(\mp0.96,\pm1.05,-0.109)$.  It has a chaotic attractor on which trajectories circulate around either $p_+$ or $p_-$ many times before making a quick transition to the other fixed point.  See Figure~\ref{G12_Attractor}.  Following~\cite{gissinger2012new}, we refer to these transitions as ``pole reversals,'' since the second component $x^2(t)$ can be thought of as a proxy for the geomagnetic dipole field, and it changes signs at these transitions.

Here, we consider Eq.~(\ref{G12}) with $\ep>0$.  We start with an initial condition near $p_+$, and after $N=100$ steps make an observation with log-likelihood $g(x) = ||x-y||^2/2$, where $x=(x^1,x^2,x^3)$.  We view $y\in\R^3$ as the outcome of a ``measurement'' made at step $N$.

We consider two cases:
\begin{itemize}

\item[{\bf Case (a):}] The measured value $y$ is near $p_-$, i.e., on the opposite ``lobe'' from the initial condition;

\item[{\bf Case (b):}] $y$ is near $p_+$, i.e., on the same ``lobe'' as the initial condition.

\end{itemize}
Figure~\ref{G12_Attractor} illustrates the initial conditions, data, and optimal paths for the two cases.  Shown are trajectories of the deterministic model (light gray), representing the chaotic attractor.  The dashed line is the most likely path with initial condition marked by ``$\bullet$'' and with measured state at time $t=10$ marked by ``$+$''; this trajectory undergoes a ``pole reversal'' (Case (a)).  The solid blue line represents the most likely path with initial condition ``$\bigcirc$'' and observation ``$\times$,'' and does not exhibit a pole reversal (Case (b)).
\begin{figure}[tb]
  \begin{center}
    \resizebox{5in}{!}{\includegraphics{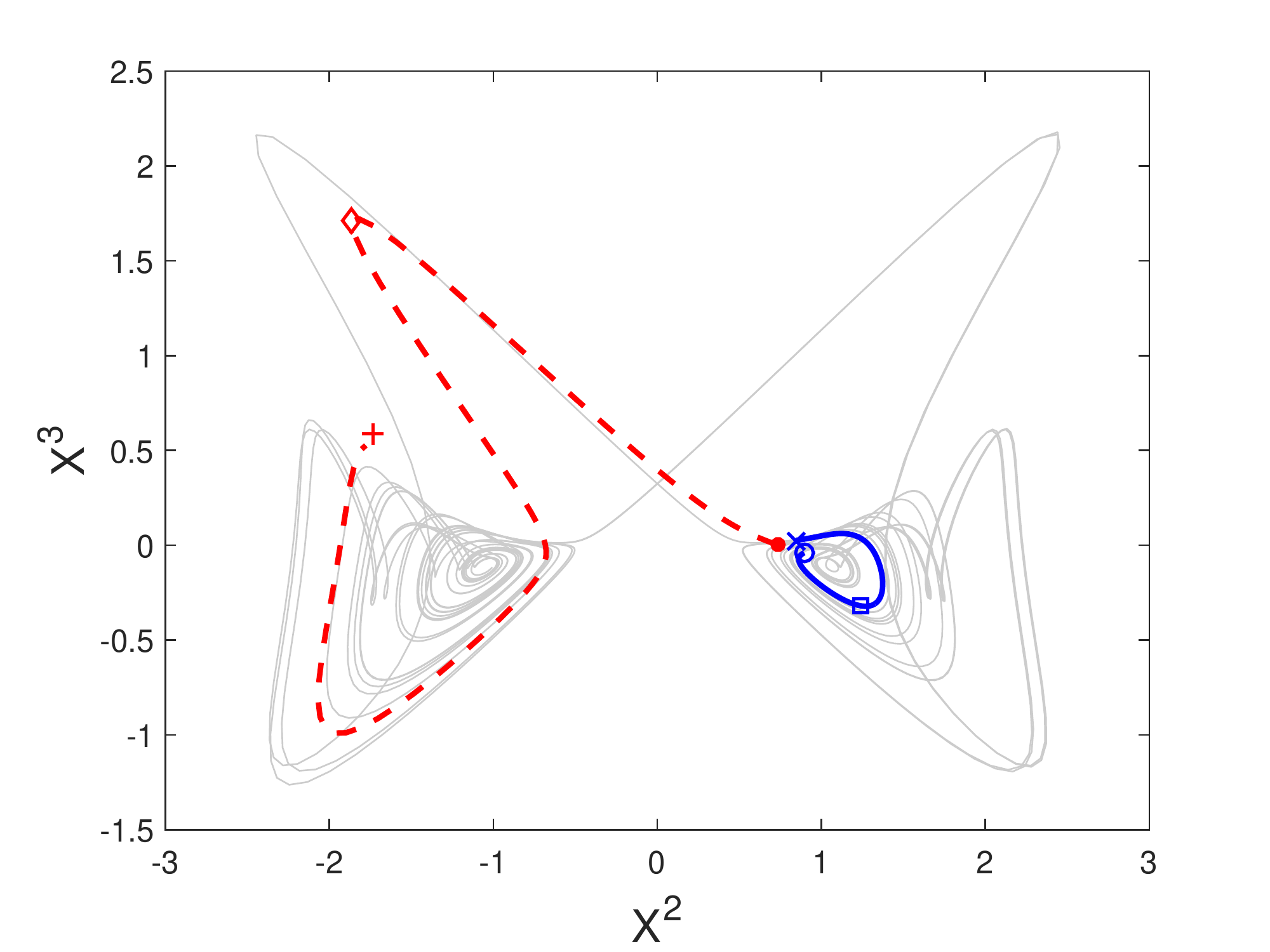}}
  \end{center}

  \caption{The Gissinger model and its phase space geometry.  Shown are trajectories of the deterministic model (light gray) projected to the $x^2$-$x^3$ plane.  The dashed line is the most likely path with initial condition marked by ``$\bullet$'' and measured state at time $t=10$ marked by ``$+$''; this trajectory undergoes a ``pole reversal'' (Case (a)).  The solid blue line represents the most likely path with initial condition ``$\bigcirc$'' and observation ``$\times$'' at $t=10$, and does not exhibit a pole reversal (Case (b)).  The symbols $\square$ and $\bigdiamond$ are the times at which we computed the histograms in Figure~\ref{sqplot_G12}.}

  \label{G12_Attractor}	
\end{figure}

To see how the two cases differ, we fix $\ep=10^{-2}$ and apply the LM and DLM to generate $1200$ sample paths in each case and plot marginals of the proposal distributions
at two different times.
In Case (a), we plot histograms of the marginal distributions at time $j \Delta t$ as marked by $\bigdiamond$ in Figure~\ref{G12_Attractor}; 
in Case (b), we plot histograms of the marginal distributions at time $j \Delta t$ as marked by $\square$.  
For each method, the resulting ``triangle plot'' consists of 
histograms of the one-dimensional marginals, $q(X^k_j|X_0)$ for $k\in\{1,2,3\}$,
and the two-dimensional marginals, $q(X^k_j,X^\ell_j|X_0),\; k\neq\ell$,
of the proposal distributions. 
The triangle plots are shown in Figure~\ref{sqplot_G12}.  In each panel, the diagonal plots are the one-dimensional marginal distributions.  The lower-triangular parts of each panel are the two-dimensional marginal distributions generated by LM, while the upper-triangular parts show marginals generated by DLM.
\begin{figure}[tb]
  \begin{center}
    \begin{tabular}{cc}
      \resizebox{3.2in}{!}{\includegraphics{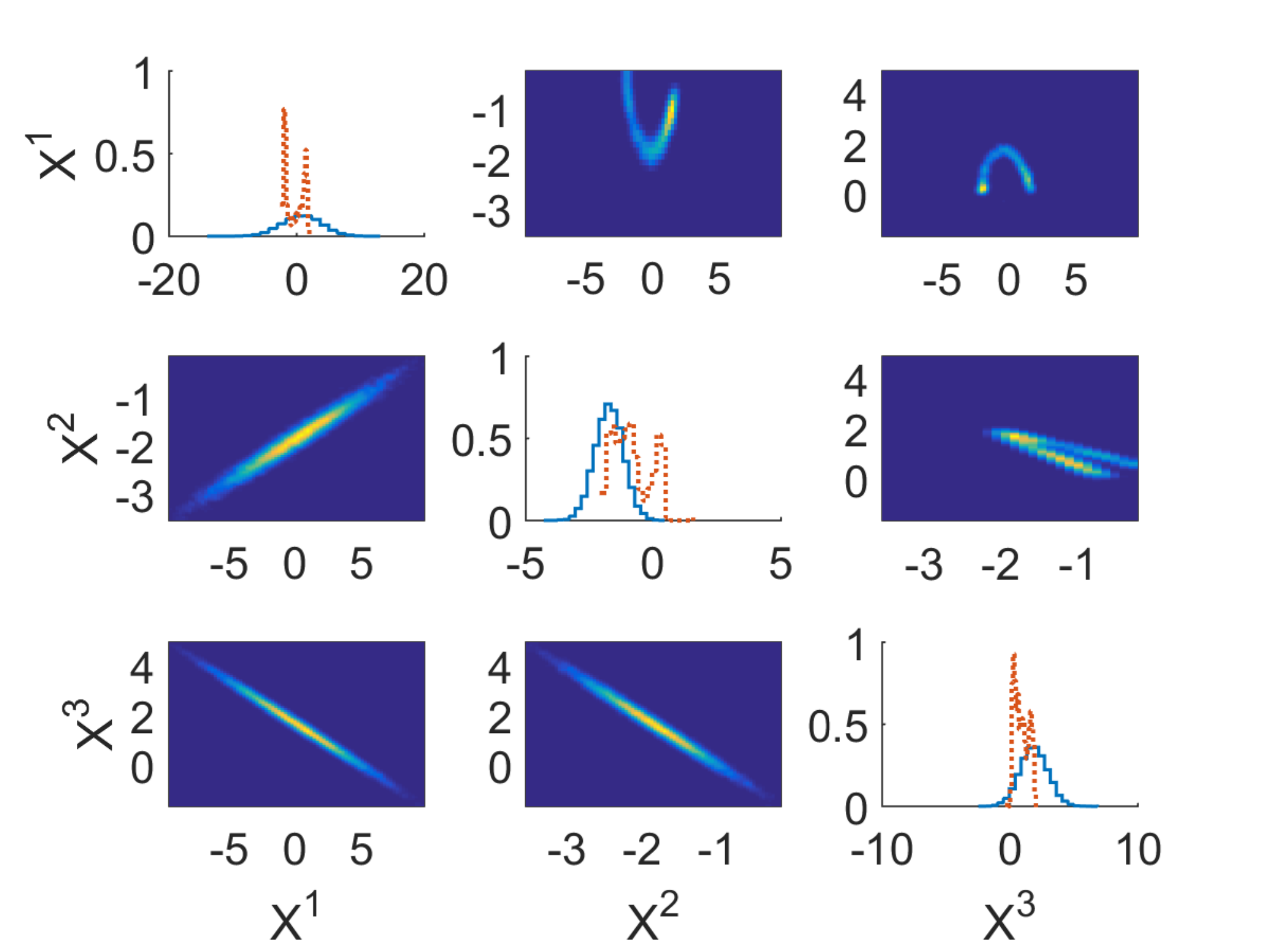}} &
      \resizebox{3.2in}{!}{\includegraphics{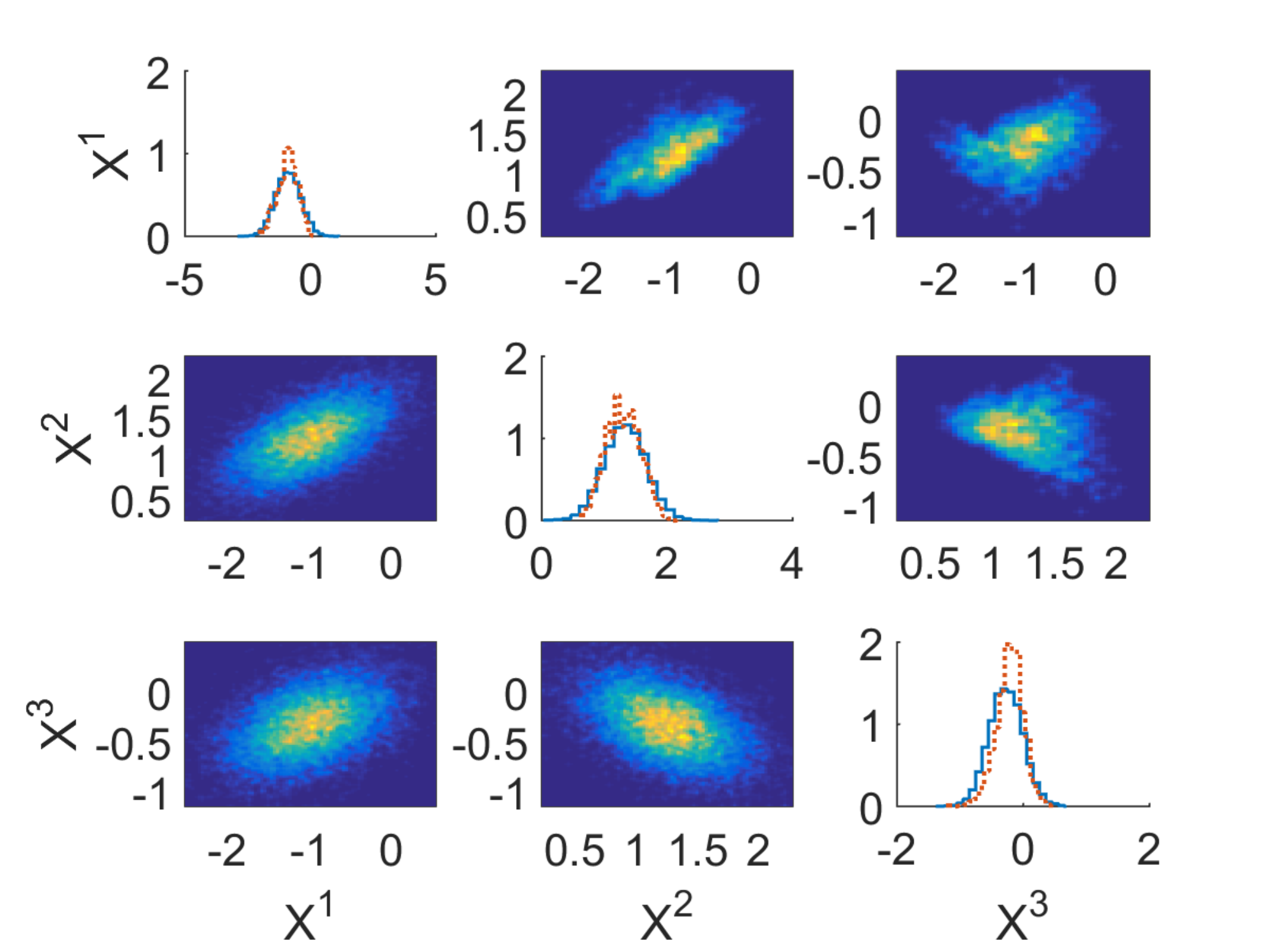}} \\
      Case (a) & Case (b) \\
    \end{tabular}
  \end{center}
  \caption{Final-time marginal distributions for the Gissinger model.  In each panel, the diagonal plots are histograms for the final-time marginal proposal distributions for of $x^1$, $x^2$, and $x^3$ (solid = LM, dashed = DLM).  The times at which the marginals are computed are marked by $\bigdiamond$ in Figure~\ref{G12_Attractor} for Case (a), and $\square$ for Case (b).  Plots on the lower-triangular submatrix are two-dimensional marginal proposal distributions computed by LM, while two-dimensional marginal proposal distributions computed by DLM form the upper-triangle
  (see text for details).}
  \label{sqplot_G12}	
\end{figure}

In Case (a), the marginal distributions of the DLM proposal are multimodal, possibly related to the underlying geometry of the strange attractor.
In contrast, the LM proposal distribution misses this complexity altogether (as one might expect).  
Moving now to Case (b), which involves starting and end points on the same lobe connected by a shorter optimal path, the marginals are unimodal, and LM and DLM give more similar answers (though there is still significant deviation from gaussianity in the DLM proposal distribution).

Finally, we vary $\ep$ in Cases (a) and (b) and apply LM, SLM, DLM and SDLM.  For each value of $\ep$, we estimate $Q$ for each of the 4 methods.  The results are shown in Figure~\ref{Q_G12_T10}.  Not surprisingly, LM breaks down for Case (a), in which the target distribution is likely multimodal.  In contrast, both DLM and SDLM exhibit the predicted scaling.  For Case (b), because the target distribution is unimodal, all four methods behave as predicted by the small noise theory.
\begin{figure}[tb]
  \begin{center}
    \begin{tabular}{cc}
      \resizebox{3in}{!}{\includegraphics{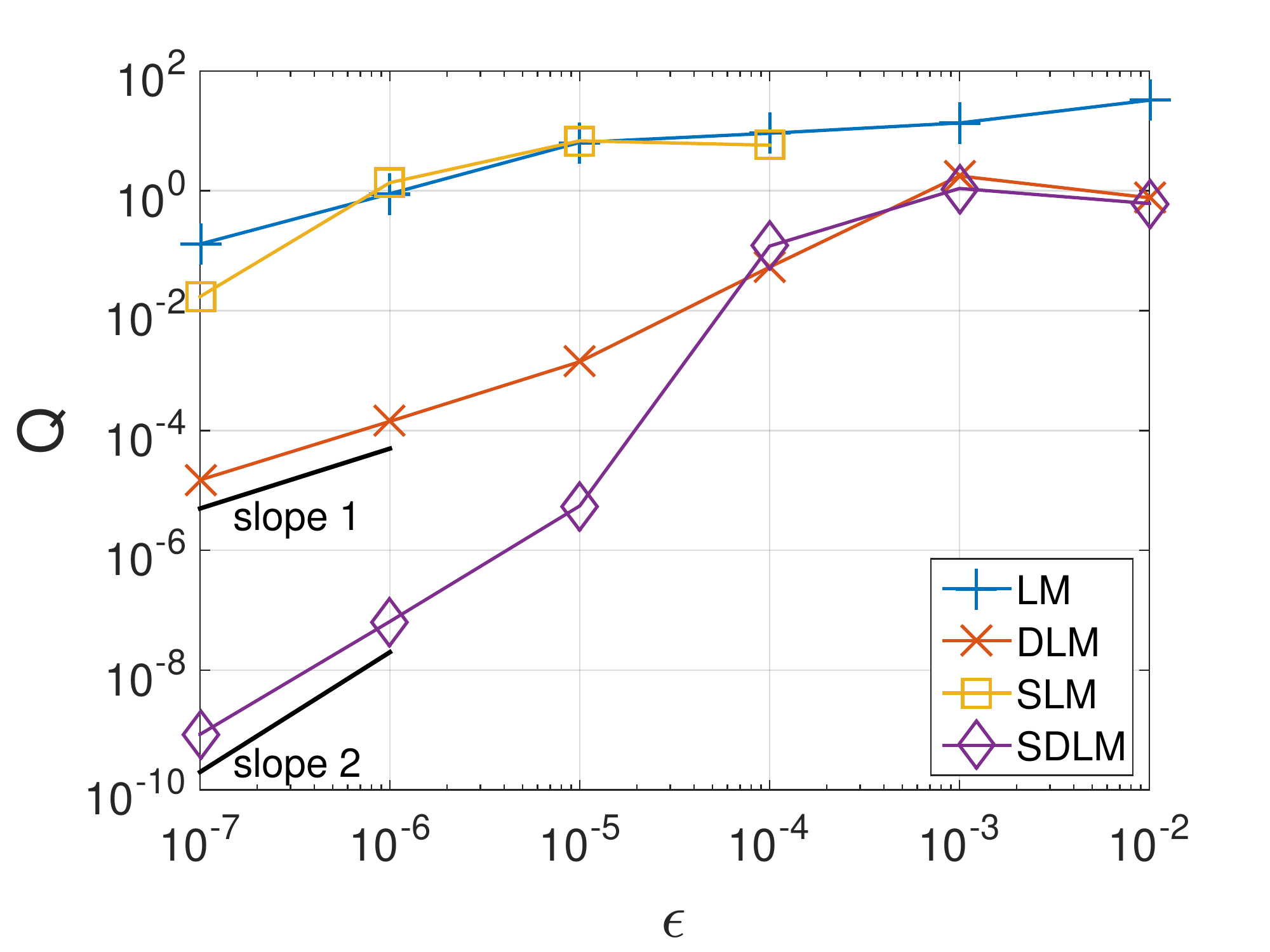}} &
      \resizebox{3in}{!}{\includegraphics{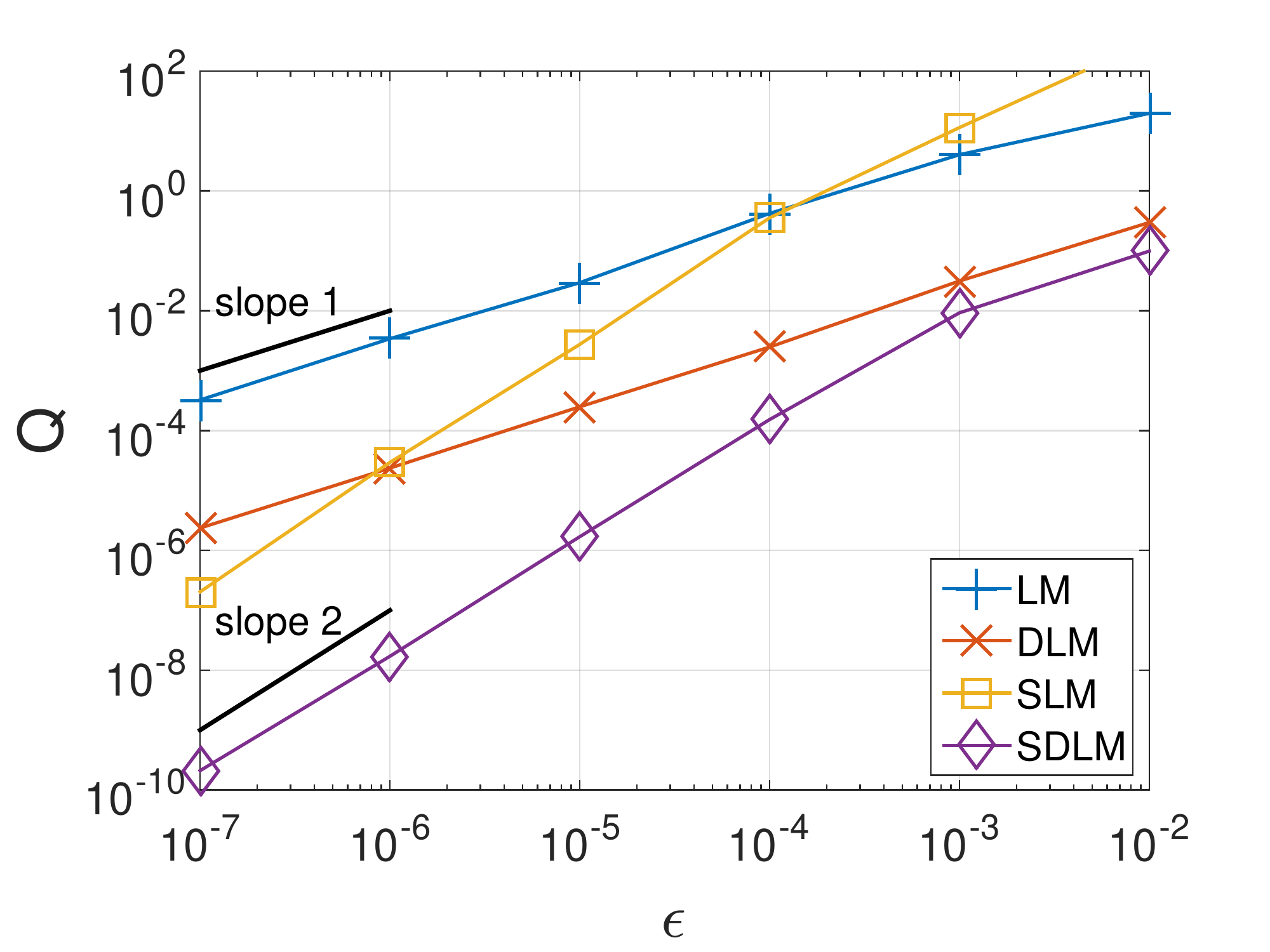}} \\
      Case (a) & Case (b) \\
    \end{tabular}
  \end{center}

  \caption{Relative variance $Q$ as a function of $\ep$ for the Gissinger model.  Case (a) involves a pole reversal, whereas Case (b) does not.}
  \label{Q_G12_T10}	
\end{figure}


\paragraph{Numerical details.}
The Gissinger model 
requires attention to numerical implementation 
when we compute its statistics.
We describe our numerical implementation in detail.
\begin{enumerate}

\item \emph{Time-stepping.} The Euler scheme for the Gissinger model requires small time steps because of numerical instabilities.  To improve stability, we discretize the drift part of Eq.~(\ref{G12}) using a standard 4th-order Runge-Kutta (RK4) method, then adding IID ${\mathcal N}(0,\sqrt{\ep}\sqrt{\dt}~I)$ normal random vectors at each step.  This yields a model of the form \eqref{SDE}, where $\tilde{f}(x,\dt)$ now represents one step of the RK4 scheme.  In all the examples shown above, the time step is $\dt = 10^{-1}$.

\item \emph{Estimation of $Q$.}  In Figure~\ref{Q_G12_T10}, because of their different variances, we use $1200$ sample paths to estimate $Q$ for DLM and for SDLM, and $12000$ paths for LM and for SLM.

\item \emph{Computing optimal paths.}  Our methods requires computing optimal paths.  For the Gissinger model, we use Newton's method.  Since explicit analytical expressions for the gradient and the Hessian are available, this is relatively straightforward to program.  To reduce the (fairly significant) computational cost of computing $\phi$ at each time step, we ``guess'' a good initialization for the optimization procedure using the solution from the previous time step using the linearized dynamics.  See \cite{leachthesis} for details.


\end{enumerate}


\section{Continuous-time limit of dynamic linear map}\label{section:continuous_time}

So far, we have focused on time discretizations of SDEs.  A natural question is what happens to the proposed algorithms in the limit $\dt\to0$.  In this section, we sketch some analytical arguments aimed at addressing these questions 
for scalar SDE. 
Though restrictive, we believe these results yield useful insights.  A more complete and rigorous analysis is left for future work, as it is expected to be more involved.


\subsection{Dynamic linear map}
For scalar SDE, the DLM can be defined through the recursion
\begin{equation}
	X^{\tdt}_{n+1} = \phi^{\tdt}_{n+1}(X^{\tdt}_{n},n) + \sqrt{\dt} ~ \sqrt{\ep}~ \sqrt{\Sigma^{\tdt}_{n+1}(X^{\tdt}_n,n)} ~\xi_n,
\end{equation}
where $\phi^{\tdt}_n(x_0,m)$, $n\in\{m,m+1,\cdots, N\}$, is the optimal path~(\ref{eq:optimal-path}) with prescribed initial condition $x_m=x_0\in\R$, $\Sigma^{\tdt}_{n+1}(X^{\tdt}_n,n)$ is the $(1,1)$th entry of the Hessian of $F^{\tdt}$ (see Eqs.~\eqref{eq:F} and \eqref{eq:DLMProposal}), and the $\xi_n$ are independent standard normal random variables.  Keeping in mind that $\phi_n(x,n)=x$ for all $n$, the above can be written as
\begin{equation}
	\label{DLMeq}
	X^{\tdt}_{n+1} = X^{\tdt}_{n} + \dt ~\frac{\phi^{\tdt}_{n+1}(X^{\tdt}_{n},n) - \phi^{\tdt}_{n}(X^{\tdt}_{n},n)}{\dt} + \sqrt{\dt} ~ \sqrt{\ep}~ \sqrt{\Sigma^{\tdt}_{n+1}(X^{\tdt}_n,n)} ~\xi_n~.
\end{equation}
%
%
Our goal in this subsection is to sketch an argument suggesting that as $\dt\to0$, solutions of \eqref{DLMeq} converge weakly~\cite{kloeden1992numerical} to those of
\begin{equation}
  \label{SDEtilde}
  d X_t = \dot{\phi}_t(X_t,t)~dt + \sqrt{\ep} ~ \sigma \cdot dB_t
\end{equation}
with $X_0=x_0$.  
Since we consider ``continuous time'' and ``discrete time'' cases, 
we mark the discrete time case by a $\Delta t$ superscript
(i.e., in this section, the function in Eq.~(\ref{eq:F}) is called $F^{\tdt}$).
In Eq.~(\ref{SDEtilde}), ``$\dot{\phi}_s(x,t)$'' denotes $\partial_s(\phi_s(x,t)),$ and the path $s\mapsto\phi_s(x_0,t)$ ($t\leq s\leq T$) minimizes the \emph{action functional}~\cite{freidlin1984random}
\begin{equation}
  \label{variation}
  F(x_{t:T}\vert x_t=x_0) = \frac{1}{2 \sigma^2} \int_t^T (\dot{x}_s-f(x_s))^2~ds + g(x_T)~,~~\phi_t(x_0,t)=x_0~.
\end{equation}
This is the continuous-time analog of Eq.~(\ref{eq:F}).  

Eq.~\eqref{SDEtilde} was derived in \cite{eijnden2012rare} as the proposal for an importance sampling algorithm.  This was later used in \cite{eijnden2013data} for data assimilation in the small-noise regime.  We assume minimizers $\phi$ of the action functional are twice-differentiable in the time parameter and satisfy the Euler-Lagrange equations; this can be justified via standard results from the calculus of variations (see, e.g., Section~3.1 of \cite{giaquinta-hildebrandt}).
%
%
In what follows, we also assume that the action functional has a single global minimum for all initial positions $x$ and initial time $t\in[0,T]$.  This \emph{unique optimal paths} assumption (the continuous-time analog of the unimodality of $p(x)$) implies that $\dot{\phi}_t(x,t)$ is defined everywhere.  Without unique optimal paths, any analysis will require more care; see, e.g., \cite{eijnden2012rare} and references therein for a discussion of these and related issues.  The assumption is natural for linear systems with unimodal likelihood functions $e^{-g/\ep}$, and may hold (approximately) in nonlinear systems when $T$ is small.


We now sketch our argument.  We begin by recalling that a numerical approximation of an SDE \emph{converges weakly with weak order $k$} if for all test functions $\psi\in C^{k+1}$ with at most polynomial growth,
\begin{equation}
  \Big|\E\big(\psi(X^{\tdt}_N)\big| X_0\big)-\E\big(\psi(X_T)\big| X_0\big)\Big| = O(\dt^k)
\end{equation}
as $\dt\to0$.  By standard results in the numerical analysis of SDEs, weak convergence is implied by ``weak consistency'' plus some mild polynomial growth conditions; see, e.g., Section~14.5 in \cite{kloeden1992numerical} for details.

In the present context, consistency means that the factors $\big(\phi^{\tdt}_{n+1}(x,n) - \phi^{\tdt}_{n}(x,n)\big)/\dt$ and $\Sigma^{\tdt}_{n+1}(x,n)$ in Eq.~(\ref{DLMeq}) approximate the corresponding factors in Eq.~(\ref{SDEtilde}) ($\dot{\phi}_t(x,n\dt)$ and $\sigma^2$, respectively).  These we now prove.

\begin{prop*}
  Under the unique optimal path assumption, we have
  \begin{itemize}
  \item[(a)]
    \begin{equation}
      \label{phicriteria}
      \frac{\phi^{\tdt}_{n+1}(x,n) - \phi^{\tdt}_{n}(x,n)}{\dt} = \dot{\phi}_{n \tdt}(x,n \dt) + \O(\dt)
    \end{equation}
    for all $n = 1,...,N$ and $x\in\R$, and
  \item[(b)]
    \begin{equation}
      \label{hcriteria}
      \Sigma^{\tdt}_{n+1}(x,n) = \sigma^2 + \O(\dt).
    \end{equation}
  \end{itemize}
\end{prop*}

%
%
%
\begin{proof}[Proof of (a)]
  We begin by proving that $\phi$ and $\phi^{\tdt}$ satisfy the first variational equations for $F$ and $F^{\tdt}$, respectively (see Eqs.~\eqref{variation} and \eqref{eq:F}).  Without loss of generality, set $t=0$ and $n=0$, and write $\phi(s):=\phi_s(x_0,0)$ for a given $x_0$.  Then the first variational equation of $F$ is the boundary value problem
  \begin{align}
    \label{eq:ContVar1}
    -\ddot{\phi}(s) + f'(\phi(s))f(\phi(s)) &=0 \\[1ex]
    \phi(0) - x(0)&= 0 \\[1ex]
    \dot{\phi}(T) - f(\phi(T)) + \sigma ~ g'(\phi(T)) &= 0
  \end{align}
  and the first variational equation for $F^{\tdt}$ is
  \begin{align}
    \label{eq:DiscVar1}
    -\frac{\phi^{\tdt}_{k-1} - 2 \phi^{\tdt}_{k} + \phi^{\tdt}_{k+1}}{\dt^2} + f'(\phi^{\tdt}_{k})f(\phi^{\tdt}_{k}) +
    \frac{f(\phi^{\tdt}_{k})-f(\phi^{\tdt}_{k-1})}{\dt} - f'(\phi^{\tdt}_{k})\frac{\phi^{\tdt}_{k+1}-\phi^{\tdt}_{k}}{\dt} &= 0 \\[1ex]
    \phi^{\tdt}_{0} - x_0 &= 0 \\[1ex]
    \frac{\phi^{\tdt}_{N}-\phi^{\tdt}_{N-1}}{\dt} - f(\phi^{\tdt}_{N-1}) + \sigma ~ g'(\phi^{\tdt}_{N}) &= 0 
  \end{align}
%
%
  By the unique optimal path assumption, Eq.~(\ref{eq:ContVar1}) is well-posed.  Eq.~(\ref{eq:ContVar1}) is equivalent to the system
  \begin{equation}
    -\dot{v} + f'(\phi)f(\phi) = 0\qquad\mbox{and}\qquad\dot{\phi} = v
  \end{equation}
  with boundary conditions $\phi(0)=0$ and $v(T)-f(\phi(T))+\sigma~g'(\phi(T))=0$, and Eq.~(\ref{eq:DiscVar1}) is equivalent to the first-order-accurate finite difference approximation
  \begin{equation}
    -\frac{v_{k}-v_{k-1}}{\dt} + f'(\phi^{\tdt}_{k})f(\phi^{\tdt}_{k}) +
    \frac{f(\phi^{\tdt}_{k})-f(\phi^{\tdt}_{k-1})}{\dt} - f'(\phi^{\tdt}_{k})~v_{k} = 0\qquad\mbox{and}\qquad
    v_k = \frac{\phi^{\tdt}_{k+1}-\phi^{\tdt}_k}{\dt}
  \end{equation}
  Convergence results for numerical approximations of two-point boundary value problems tell us that for first-order accurate finite difference schemes, point-wise errors are uniformly bounded by $C\dt$ for some $C>0$ (see, e.g., \cite{keller_2ptbvp} and references therein).  In particular, we have $(\phi^{\tdt}_{n+1}-\phi^{\tdt}_n)/\dt=v_n=\dot{\phi}(n\dt) + O(\dt)$ for each $n$, as claimed.
\end{proof}

\begin{proof}[Proof of (b)]
  To prove \eqref{hcriteria}, we consider the second variational equations of $F$ and $F^{\tdt}$.  For $F$, we obtain a Sturm-Liouville boundary value problem
  \begin{align*}
    (Lu)(s)& =0 \\
    u(0) &= 0 \\
    u'(T) + (-f'(\phi(s)) + \sigma ~ g''(\phi(T)))u(T) &= 0
  \end{align*}
  where the operator $L$ is defined by
  \begin{align*}
    Lu = -u''(s) + (f'(\phi(s))^2 + f''(\phi(s))f(\phi(s)))u(s),
  \end{align*}
  $\phi$ is the solution to the first variational equation, and $u$ is a test function.  The second variational equation for $F^{\tdt}$ is
  \begin{align*}
    (H/\dt) u^{\tdt}&=0\\
    u^{\tdt}_{0} &= 0\\
    \frac{u^{\tdt}_{N}-u^{\tdt}_{N-1}}{\dt} -f'(\phi^{\tdt}_{N-1})u^{\tdt}_{N-1} + \sigma ~ g''(\phi^{\tdt}_{N})u^{\tdt}_{N} &= 0
  \end{align*}
  where $H$ is the Hessian of $F^{\tdt}$, and
  \begin{align*}
    (H/\dt) u^{\tdt}=&	 -\frac{u^{\tdt}_{k-1} - 2u^{\tdt}_k + u^{\tdt}_{k+1}}{\dt^2} + (f'(\phi^{\tdt}_{k})^2 +  f(\phi^{\tdt}_{k})\cdot f''(\phi^{\tdt}_{k}))u^{\tdt}_{k} \\
    &+ \frac{f'(\phi^{\tdt}_{k}) u^{\tdt}_{k}-f'(\phi^{\tdt}_{k-1}) u^{\tdt}_{k-1}}{\dt} -\frac{u^{\tdt}_{k+1}-u^{\tdt}_{k}}{\dt}\cdot f'(\phi^{\tdt}_{k}) - \frac{\phi^{\tdt}_{k+1}-\phi^{\tdt}_{k}}{\dt} \cdot f''(\phi^{\tdt}_{k}))u^{\tdt}_{k}.
  \end{align*}
  Note that the discrete equations 
  can also be obtained by applying a first order discretization scheme
  to the continuous equations.

  The differential operator $L$ has an associated Green's function
  \begin{align*}
    K(t,s) = \frac{1}{y'_1(0)y_2(0)}
    \begin{cases}
      y_1(t)y_2(s) & 0<t<s \\
      y_2(t)y_1(s) & 0<s \leq t
    \end{cases}
  \end{align*}
  where $y_1$ is a solution that satisfies the left Dirichlet boundary condition, 
  while the solution $y_2$ satisfies the mixed boundary condition on the right.
  The analog of the Green's function for the discretized problem is $H^{-1}$.
  Specifically, the first element of the first row of $H^{-1}$ 
  is a second order approximation of the Green's function $K(\dt,\dt)$:
  \begin{equation}
    (H^{-1})_{1,1} = K(\dt,\dt) + \O(\dt^2) ~.
  \end{equation}
  A Taylor expansion of $K$ at the origin gives
  \begin{equation}
    K(\dt,\dt) =  \sigma^2\dt + \dt^2\frac{y'_2(0)}{y_2(0)} + \O(\dt^3)~,
  \end{equation}
  Combined, we thus have
  \begin{equation}
    (H^{-1})_{1,1} = \sigma^2 \dt + \O(\dt^2)
  \end{equation}
  Since $\Sigma^{\tdt}_{n+1}(x,n) = (H^{-1})_{1,1}/\dt$, this shows that $\Sigma^{\tdt}_{n+1}(x,n)= \sigma^2 + \O(\dt)$.
\end{proof}

\subsection{Small noise analysis for the continuous-time limit of DLM}
We investigate how the efficiency of the dynamic linear map, as measured by the quantity $Q$ (see Eq.~\eqref{eq:Q}), is affected by taking the $\dt\to 0$ limit, and apply the theory presented in \cite{spiliopoulos2015nonasymptotic} to show that $Q$ scales linearly in the small noise parameter~$\ep$ even as $\dt\to 0$.

First, we note that the weights of the continuous limit of the 
DLM follow from the Cameron-Martin-Girsanov Theorem
\cite{freidlin1984random}
\begin{equation}
	w(X) \propto \exp \bigg( -\frac{1}{\sqrt{\ep}}\int_0^T v(X_s,s) \cdot dB_s -\frac{1}{2\ep} \int_0^T v(X_s,s)^2 ds -\frac{1}{\ep}g(X_T) \bigg)
\end{equation}
where $v(x,t) = \sigma^{-1} \cdot (\phi'_t(x,t)-f(x))$.
The relative variance of the weights can be written as
\begin{equation}
	\label{Qpre}
	Q = e^{-(V(0,x_0)-2G(0,x_0))/\ep}-1,
\end{equation}
where
\begin{align*}
	G(x,t) &= -\ep \log(\E_q[w|x_t = x])	\\
	V(x,t) &= - \ep \log(\E_q[(w)^2|x_t = x])
\end{align*}
In \cite{spiliopoulos2015nonasymptotic},
it was shown that $V$ can be expanded in powers of $\ep$
when the minimizer $\phi$ of \eqref{variation} is unique for all $(x,t)$ in the domain.
A calculation shows that $G$ can also be expanded in powers of $\ep$,
with similar coefficients.
In summary, we have
\begin{align}
\label{eq:G}
	G(x,t) &=  G_0(x,t) + \ep \cdot G_1(x,t) + \ep^2 \cdot G_2(x,t) + \O(\ep^3) \\ 
\label{eq:V}
	V(x,t) &=  V_0(x,t) + \ep \cdot V_1(x,t) + \ep^2 \cdot V_2(x,t) + \O(\ep^3)
\end{align}
where the coefficients $G_i, V_i$, $i=0,1,2$, satisfy the following system of PDEs:
\begin{align*}
\partial_tG_0+f\partial_x G_0-\frac{\sigma^2}{2}(\partial_x G_0)^2&=0,\quad G_0(x,T) = g(x)\\
\partial_tV_0+(f+\sigma^2\partial_xG_0) \cdot \partial_xV_0-
\frac{\sigma^2}{2} (\partial_xV_0)^2-\sigma^2 (\partial_xG_0)^2 &=0,\quad V_0(x,T) = 2g(x)\\
\partial_tG_1+f \cdot \partial_xG_1+\frac{\sigma^2}{2}\partial_{xx}G_0-\sigma^2\partial_xG_0 \cdot \partial_xG_1&=0,\quad G_1(x,T)=0\\
\partial_tV_1+(f+\sigma^2\partial_xG_0) \cdot \partial_xV_1+\frac{\sigma^2}{2}\partial_{xx}V_0-\sigma^2\partial_xV_0 \cdot \partial_xV_1&=0,
\quad V_1(x,T)=0\\
\partial_tG_2+f \cdot \partial_xG_2+\frac{\sigma^2}{2}\partial_{xx}G_1-\sigma^2\partial_xG_0 \cdot \partial_xG_2 -\frac{\sigma^2}{2} (\partial_xG_1)^2 &=0,\quad G_2(x,T)=0\\
\partial_tV_2+(f+\sigma^2\partial_xG_0) \cdot \partial_xV_2+\frac{\sigma^2}{2}\partial_{xx}V_1-\sigma^2\partial_xV_0 \cdot \partial_xV_2 -\frac{\sigma^2}{2} (\partial_xV_1)^2 &=0,
\quad V_2(x,T)=0.
\end{align*}
(These equations are similar in structure to those of the WKB approximation~\cite{bender-orszag}, with the leading order term given by a nonlinear PDE of Hamilton-Jacobi type and a hierarchy of linear transport equations for the higher-order terms.)  One can check that $V_0=2G_0$ and $V_1=2G_1$, but $V_2\neq 2G_2$.  Combining the expansions \eqref{eq:G} and \eqref{eq:V} we thus have
\begin{equation}
\label{eq:DiffTerm}
	V(x_0,0)-2G(x_0,0) = \ep^2 K_2 + 	O(\ep^3),
\end{equation}
where $K_2 = V_2-2G_2$ satisfies
\begin{align}
	\label{K2PDE}
	\partial_t K_2 + f \cdot \partial_x K_2 - \sigma^2 \partial_x G_0 \cdot \partial_x K_2 - \sigma^2 (\partial_x G_1)^2 &= 0,
	\quad K_2(x,T) = 0.
\end{align}	
Using \eqref{eq:DiffTerm}  in 
the expression of the relative variance $Q$ in \eqref{BM_WHist}, 
and expanding in $\ep$ results in
\begin{equation}
	Q = \ep \cdot K_2(x_0,0) + \O(\ep^2).
\end{equation}
Thus, the performance criterion $Q$ for this continuous time method
scales linearly with $\ep$.

\section{Concluding discussion}
\label{sec:Conclusions}

In this paper, we study a class of importance samplers for SDEs designed for data assimilation tasks in the small (observation and dynamic) noise regime.  We have extended a small noise analysis for implicit samplers~\cite{goodman2015small} to importance sampling for SDEs.  We have also shown that a symmetrization procedure, originally proposed in~\cite{goodman2015small}, can be applied effectively to obtain higher-order samplers for SDEs.  Moreover, we have shown that a dynamic version of the importance sampler retains the same asymptotic performance but is more robust in problems with multimodal distributions.

\medskip

Our work also points to a number of directions for future research:
\begin{enumerate}

\item\emph{Multimodal distributions.}  Our analysis is limited to unimodal target distributions, but multimodal distributions do occur in practice.  We believe an analysis for such problems (which necessarily means dealing with $q(x_{n+1}|x_n)$ with jump discontinuities), possibly on concrete examples, would yield useful insights into the performance of DLM in more general situations than the ones examined here.
%
%
One use for such an analysis is to compare DLM with other data assimilation methods, e.g., the ensemble Kalman filter, which may require less computation in nearly gaussian problems.


\item\emph{Continuous time limits.} 
In discrete time, the dimension of the sampling problem we consider is equal to the dimension of a discretized path of an SDE and, thus, equal to the product of the state dimension and the number of time steps of the path.  
Our continuous time limit of the DLM for scalar SDE indicate that 
a large dimension due to a small time step is unproblematic,
but our results do not indicate how the efficiency of DLM degrades when the dimension of the SDE is large.

\item\emph{Symmetrization in continuous time.}  
Our results with symmetrized methods in discrete time are encouraging, but we currently do not have theoretical results on symmetrization in continuous time.

\item\emph{Long timescales.} As mentioned in the Introduction, the methods discussed in this paper bear a close resemblance to methods proposed in \cite{eijnden2012rare} and \cite{dupuis2007subsolutions} for rare event simulation.  However, in this paper we have assumed a fixed final time $T$, whereas for many (if not most) rare event problems of interest, the relevant timescale tends to $\infty$ as $\ep\to0$ (e.g., $T=O(1/\ep)$), and our methods are not expected to perform well on such long time scales.  It would be of theoretical and practical interest to extend the ideas described here to the setting of rare event simulation, particularly the idea of symmetrization.

\item\emph{Problems that do not come from SDEs.} Also mentioned in the Introduction is the possibility of extending the methods proposed here, in particular symmetrization, to more general sequential Monte Carlo sampling problems.

\end{enumerate}

\section{Acknowledgments}

KL and AL were supported in part by NSF grant DMS-1418775.  MM was supported by NSF grant DMS-1619630, the Office of Naval Research (grant number N00173-17-2-C003), and by the Alfred P.~Sloan Foundation.  The authors thank Profs.~Jonathan Goodman, Jonathan Weare, and Kostas Spiliopoulos for many helpful conversations and some of the references.

\bibliographystyle{siam}
\bibliography{Symmetrizing_Samplers}

\end{document}